\documentclass[a4paper]{amsart}
\usepackage[english]{babel} 

\usepackage{epsfig,amsmath}
\usepackage{xspace}
\usepackage[psamsfonts]{amssymb}
\usepackage[latin1]{inputenc}
\usepackage{float}\restylefloat{figure}
\usepackage{color}

\renewcommand{\epsilon}{\varepsilon}
\renewcommand{\emptyset}{\varnothing}

\newcommand{\wt}[1]{\widetilde{#1}}
\newcommand{\setof}[2]{\big\{{#1}\,\big|\,{#2}\big\}}

\def\bEA{\begin{eqnarray*}}
\def\eEA{\end{eqnarray*}}

\def\ds{\displaystyle}
\def\tend{\longrightarrow}
\def\ds{\displaystyle}

\def\on{\operatorname}
\def\cal{\mathcal}

\def\C{{\mathbb C}}
\def\D{{\mathbb D}}

\def\N{{\mathbb N}}

\def\R{{\mathbb R}}

\def\Z{{\mathbb Z}}

\def\Re{{\on{Re}\,}}

\usepackage{ogonek}
\newcommand{\Sw}{\'Swi\k{a}tek\xspace}

\usepackage{enumerate}
\usepackage[pdftex]{hyperref}
\usepackage{tikz}
\usetikzlibrary{arrows}
\usepackage{enumitem}

\newtheorem{lemma}{Lemma}
\newtheorem{definition}[lemma]{Definition}

\newtheorem*{sublemma}{Sublemma}

\newtheorem{theorem}{Theorem}
\newtheorem{corollary}[lemma]{Corollary}
\newtheorem{proposition}[lemma]{Proposition}
\newtheorem*{theorem*}{Theorem}
\newtheorem*{conjecture}{Conjecture}
\newtheorem*{complem}{Special Case}
\newtheorem*{quest}{Question}
\newtheorem{notation}[lemma]{Notation}

\newcommand{\ov}{\overline}

\newcommand{\fil}[1]{\on{fill}({#1})}
\newcommand{\realm}[1]{\widehat{#1}}
\newcommand{\apo}[1]{\widetilde{#1}}

\begin{document}


\title{Herman's condition and Siegel disks of bi-critical polynomials}

\author{Arnaud Chéritat}
\address{(Centre National de la Recherche Scientifique)\\
Institut de Mathématiques de Bordeaux\\
351, cours de la Libération\\
F 33405 Talence cedex\\
France}
\email{arncheritat@math.u-bordeaux.fr}

\author{Pascale Roesch}
\address{Institut de Mathématiques de Marseille I2M\\
Aix-Marseille Université, Technopôle Château-Gombert\\
39, rue F.\ Joliot Curie\\
13453 Marseille Cedex 13\\
France}
\email{pascale.roesch@univ-amu.fr}

\date{\today}

\abstract We extend a theorem of Herman from the case of unicritical polynomials to the case of polynomials with two finite critical values. This theorem states that Siegel disks of such polynomials, under a diophantine condition (called Herman's condition) on the rotation number, must have a critical point on their boundaries.
\endabstract

\maketitle


\section{Introduction}

By a \emph{Siegel disk} of a rational map $f$ of degree at least $2$ we mean a maximal domain on which an iterate of $f$ is conjugate to an irrational rotation of a disk.
One can wonder what phenomena at the boundary of a Siegel disk prevent $f$ from having a larger linearization domain. Obviously, such a domain cannot contain a critical point.
A theorem of Fatou asserts that for an attracting periodic point, there is always a linearization domain that extends up to at least one critical point.

\begin{quest}Does the boundary of a Siegel disk always contain a critical point? \end{quest}
The answer is no. Ghys and Herman gave the first examples of polynomials having a Siegel disk without a critical point on the boundary (see~\cite{G2,H4}).

Let us now introduce the following subset of the irrationals.
\begin{definition}
Let $\cal H$ be the set of real numbers $\theta\in\R$ such that every orientation-preserving analytic circle diffeomorphism of rotation number\footnote{For further information on the notion of \emph{rotation number}, see for instance \cite{Milnor} or \cite{dMvS}, Chapter~I, Section~1. 
} $\theta$ is analytically conjugate to a rotation.
\end{definition}
Herman proved that $\cal H$ is non-empty by showing that it contains all diophantine numbers \cite{H2}.
Yoccoz characterized numbers in $\cal H$ in terms of their continued fraction expansion \cite{Y1}.

\begin{theorem*}[Ghys, \cite{G2}] For every rational map $f$ of degree $\geq 2$ having a Siegel disk $\Delta$ of period one with rotation number in $\cal H$, if $\partial \Delta$ is a Jordan curve then it contains a critical point.\end{theorem*}
Later, Herman proved in \cite{H1} a general extension theorem for holomorphic maps having an invariant annulus with rotation number in $\cal H$.
Among the corollaries he obtained is the following:
\begin{theorem*}[Herman]\ 
\begin{enumerate}
\item\label{herman:item:1} Suppose $f$ is a rational map of degree $\geq 2$ having a Siegel disk $\Delta$ of period one with rotation number in $\cal H$. Then $f$ cannot be injective in any neighborhood of $\partial \Delta$.
\item\label{herman:item:2} Every unicritical polynomial $f(z)=z^d+c$ having a Siegel disk $\Delta$ of period one with rotation number in $\cal H$ has a critical point on $\partial \Delta$. 
\end{enumerate}
\end{theorem*}

\begin{remark}\label{rem:one}
Let $K$ be a compact subset of the domain of definition of a holomorphic map $f$. The following are equivalent:
\begin{enumerate}[label=(\roman*)]
\item There is no neighborhood of $K$ on which $f$ is injective.
\item\label{item:rem:ii} Either $f$ has a critical point on $K$ or the restriction of $f$ to $K$ is non-injective (or both).
\end{enumerate}
This applies in particular to $K=\partial \Delta$ in case~\eqref{herman:item:1} of Herman's theorem.
\end{remark}

The corresponding result for Siegel disks of higher periods is the following:

\begin{theorem*}[Herman's Theorem in the periodic case]\label{cor:p}\ 
\begin{enumerate}
\item\label{item:cor:p:1} For every rational map $f$ of degree $\geq 2$ having a Siegel disk $\Delta$ of period $p$ with rotation number in $\cal H$, there exists an $i$ between $0$ and $p-1$ such that $f$ is not injective in any neighborhood of $\partial f^i(\Delta)$.
\item\label{item:cor:p:2} For every unicritical polynomial $f(z)=z^d+c$ having a Siegel disk $\Delta$ of period $p$ and with rotation number in $\cal H$, there exists an $i$ between $0$ and $p-1$ such that $f$ has a critical point on $\partial f^i(\Delta)$.
\end{enumerate}
\end{theorem*}
\noindent For completeness, we have included a proof of this corollary in Section~\ref{sub:sup1}.

\medskip

One may state 
the following:\footnote{Stronger conjectures have been formulated: some are mentioned in Sections~\ref{sec:rem} and~\ref{sec:state}.}
\begin{conjecture} The boundaries of Siegel disks of rational maps contain a critical point whenever the rotation number is in $\cal H$.
\end{conjecture}

We contribute a further step towards this conjecture by proving the following result:

\begin{theorem}[Main theorem]\label{thm:main}
For every polynomial with two finite critical values, and a Siegel disk $\Delta$ of arbitrary period and of rotation number in $\cal H$, there is an element in the cycle of $\Delta$ whose boundary contains a critical point.
\end{theorem}

Our proof is a refinement of Herman's proof with added ingredients supplied by two theorems of Mañé as well as the separation theorem of Goldberg--Milnor--Poirier--Kiwi.

\subsection{Remarks}\label{sec:rem}

There are many polynomials with only two critical values besides those with two critical points. Up to conjugacy, they come in infinitely many two parameter families, one for each planar finite bipartite tree with at least one branch point of each color.

The boundary $\partial \Delta$ of a Siegel disk of a rational map $f$ of degree $\geq 2$ is locally connected if and only if it is a Jordan curve (see \cite{Milnor}, Lemma~18.7), and this is in fact valid for general holomorphic maps as soon as the Siegel disk is compactly contained in the domain of definition. In this case, the restriction of $f$ to $\partial \Delta$ is necessarily injective (see also \cite{Milnor}, Lemma~18.7). Hence if we further assume that the rotation number is in $\cal H$, we recover Ghys' result from part \eqref{herman:item:1} of Herman's theorem: there is a critical point on $\partial \Delta$.

There are many rational maps (including many polynomials) whose Siegel disk boundaries are known to be Jordan curves. There are no known examples of rational maps where this property fails. However there are examples of holomorphic maps defined on a simply connected open subset $U$ of $\C$ with a Siegel disk compactly contained in $U$ whose boundary is not locally connected (for instance a pseudocircle, see \cite{C1}). Moreover these maps can be chosen to be injective. As it has been conjectured by Douady, it may well be the case that all Siegel disks of polynomials (or rational maps) are always Jordan domains. If true then Ghys' theorem will imply that every Siegel disk of a rational map of degree $\geq 2$ with rotation number in $\cal H$ has a critical point on its boundary.
However, this conjecture of Douady is still open and seems out of reach today.

Note also that Herman did \emph{not} prove part \eqref{herman:item:2} of his theorem by showing that $f$ must be injective on $\partial \Delta$. For instance, it is still unknown today whether there exist rotation numbers $\theta$ (even under the restriction $\theta\in\cal H$, or under the restriction $\theta\notin\cal H$) for which the boundary of the Siegel disk of $z\mapsto e^{2\pi i\theta} z+z^2$ is the whole Julia set, even though this now seems unlikely in light of the work of Inou and Shishikura \cite{IS}. Another scenario for non-injectivity is illustrated in Figure~\ref{fig:b}.

Recall that the filled-in Julia set is connected if and only if the Julia set is connected if and only if no critical orbit tends to infinity.
If a critical orbit escapes, then our 
main theorem follows from Herman's.
Indeed there is then a polynomial-like restriction of $P^p$ whose filled-in Julia set is connected and contains $\Delta$.\footnote{A proof of this well-known fact is included here for completeness: see Lemma~\ref{lem:bh}.} Moreover, this restriction has only one critical value, so it is conjugate to a unicritical polynomial near its Julia set. We can then apply part~\eqref{herman:item:2} of Herman's theorem. Therefore we could restrict to polynomials with connected Julia sets, but the proof that we will give does not require this assumption.

\begin{figure}
\begin{picture}(300,200)
\put(0,0){\includegraphics[width=300pt]{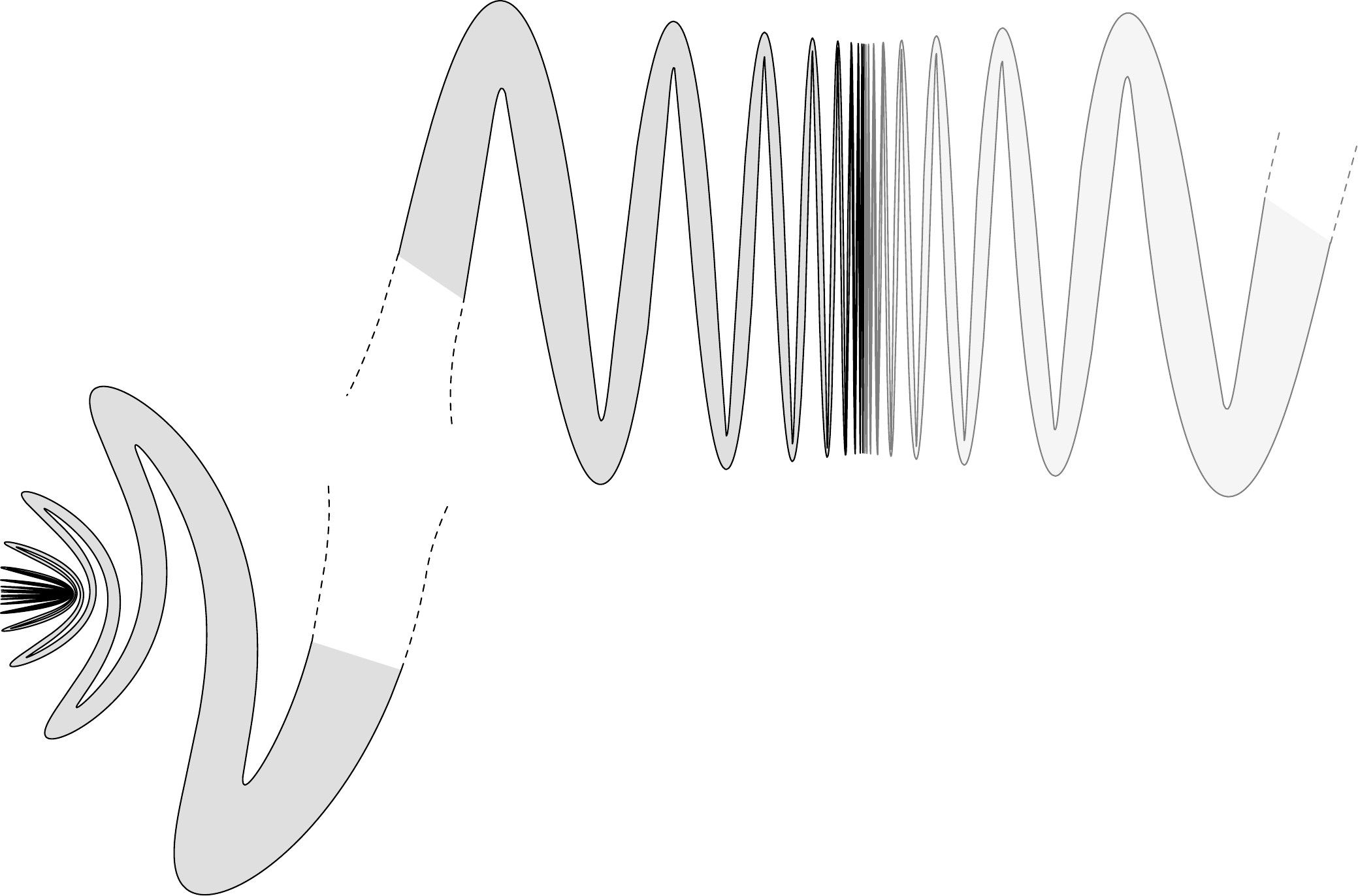}}
\put(83,94){$\Delta$}
\put(290,180){$\Delta'$}
\put(190,72){critical point}
\put(190,87){\begin{tikzpicture}\draw[-triangle 60] (1,-2) -- (0,0);\end{tikzpicture}}
\put(120,29){critical value}
\put(14,32){\begin{tikzpicture}\draw[-triangle 60] (3.5,-1.2) -- (0,0);\end{tikzpicture}}
\end{picture}
\caption{We cannot yet rule out the existence of a quadratic polynomial $P(z)=z^2+c$ with a Siegel disk $\Delta$ of period one with the critical point belonging to the boundary $\partial \Delta$, but with the restriction of $P$ to $\partial \Delta$ non-injective. The image above gives one scenario. Of course this picture is not complete, since infinitely many copies of the sine-like curves (images and preimages) must also appear on $\partial \Delta$. The first pre-image of $\Delta$ is the union of $\Delta$ and $\Delta'=$ the image of $\Delta$ by the symmetry $z\mapsto -z$ of $P$. }
\label{fig:b}
\end{figure}
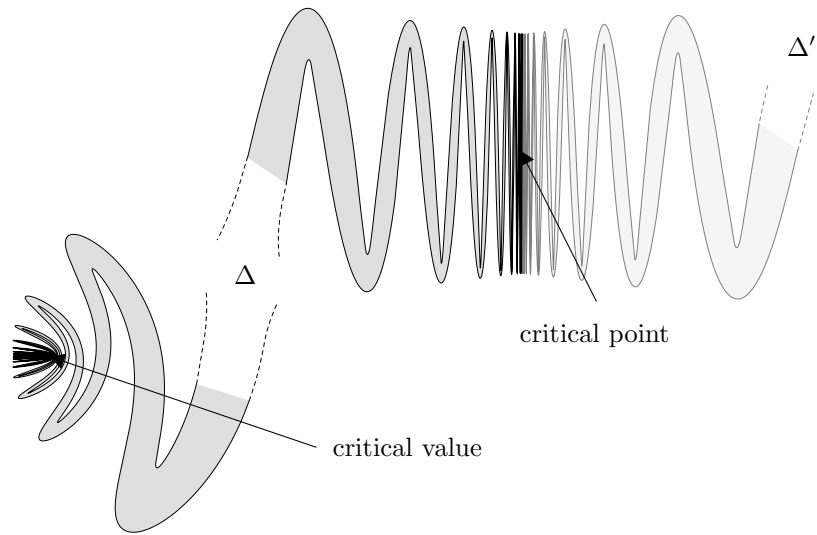

\subsection{More about our knowledge}\label{sec:state}

With a different method than Herman's (approaching the 
boundary from inside instead of 
outside, and using the Schwarzian derivative), Graczyk and \Sw \ proved in \cite{GrSw}:
\begin{theorem*}[Graczyk and \Sw] If a Siegel disk has bounded type rotation number and is compactly contained in the domain of definition of the map, then its boundary contains a critical point.
\end{theorem*}

In particular for rational maps (including polynomials) of degree $\geq 2$, there is always a critical point on the boundary of bounded type Siegel disks. This also follows from quasiconformal surgery techniques \cite{Zhang}, which prove more:

\begin{theorem*}[Zhang]
The boundary of a Siegel disk of a rational map of degree $\geq 2$ with bounded type\footnote{For polynomials, an extension of this theorem has recently been announced by Zhang in~\cite{Zhang2}. It uses trans-quasiconformal surgery as in~\cite{PZ} and covers rotation numbers of type $PZ$, i.e.\ irrationnals whose continued fraction expansion $[a_0;a_1,\ldots]$ satisfies $\log a_n=\cal O(\sqrt n)$ as $n\to+\infty$. These numbers form a subset of full Lebesgue measure of $\R$.} rotation number is a Jordan curve. Thus it contains at least one critical point by Ghys' theorem.
\end{theorem*}

In the case of a polynomial with a Siegel disk $\Delta$ whose boundary is not a Jordan curve, $\partial \Delta$ could separate the plane into more than two components. As already noted in Section~\ref{sec:rem}, is not known if such cases occur for polynomials or even for rational maps, and the common belief is that they do not. In fact there are no known examples of a holomorphic map $f:U\to\C$ with $U\subset \C$, having a fixed Siegel disk that is compactly contained in $U$ and whose boundary separates the plane into more that two connected components (the examples of \cite{C1} do not). If there is such a component, then $\partial\Delta$ cannot be locally connected, (it is even worse, at least for polynomials, see the second theorem of Rogers below). Examples of simply connected open sets with boundaries having more that two complementary components can be found on Figure~\ref{fig:cases}.

Let us define the \emph{filled-in Siegel disk} $\realm \Delta$ as the union of $\partial\Delta$, together with all bounded connected components of $\C\setminus\partial\Delta$ (see also Definition~\ref{def:fill}).

Working towards a generalization of part~\eqref{herman:item:2} of Herman's theorem to higher degree polynomials, Rogers proved the following in \cite{Rogers}. In the present article, we reprove it.

\begin{theorem*}[Rogers] If $f$ is a polynomial with a Siegel disk $\Delta$ of period one and rotation number in $\cal H$, then there is a critical point in the filled-in Siegel disk $\realm{\Delta}$.
\end{theorem*}
\noindent The critical point in the theorem of Rogers above cannot be in the Siegel disk $\Delta$, so it must be either on $\partial\Delta$ or in another bounded component of $\C\setminus\partial\Delta$.

Rogers also proved the following theorem in \cite{Rogers}, that we will not use in the present article:
\begin{theorem*}[Rogers]
If the polynomial f has a Siegel disk $\Delta$ of period one and with rotation number in $\cal H$, then either the boundary
$\partial \Delta$ of $\Delta$ contains a critical point
or $\partial \Delta$ is an indecomposable continuum with three properties: (1) $\partial \Delta$ has at least three
complementary domains, and $\partial \Delta$ is the boundary of each of its complementary domains, (2) each bounded
complementary domain of $\partial \Delta$ is a component of the grand orbit of $\Delta$ and so a bounded component of the Fatou set, (3) one of the bounded complementary domains of $\partial \Delta$
contains a critical point.
\end{theorem*}

To clarify the content of the above theorem, recall that a \emph{continuum} is a non-empty compact connected metric space.
A continuum is \emph{indecomposable} if it cannot be written as the union of two closed connected proper subsets.\footnote{The usual mistake is to confuse this definition with that of connectedness: it is not assumed that the two closed subsets are disjoint, instead they are assumed to be connected.}
An indecomposable continuum is never locally connected.
An indecomposable continuum in the plane which is the common boundary of at least three complementary domains is called a \emph{Lakes of Wada continuum}.
\medskip

The proofs in the present article come in three cases, as described in Section~\ref{s:reduction}. Our proof of the second case (in Lemma~\ref{lem:precii}) is a close relative of Rogers' work. 
\medskip

To conclude this section let us mention that, based on the present work, our theorem has been extended to the case of transcendental entire maps by Benini and Fagella in~\cite{BF}.

\section{Preliminaries}

\subsection{Filled-in sets and their properties}\label{sec:fill}

\begin{definition}\label{def:fill}
Given a compact subset $X$ of $\C$, let $\fil{X}$ be the union of $X$ and all bounded components of $\C\setminus X$. 
\end{definition}
The set $\fil{X}$ is also the complement of the unbounded connected component of the complement of $X$. A compact subset $X$ of $\C$ is called \emph{full} if $X=\fil{X}$. The set $\fil{X}$ is full, i.e.\ $\fil{\fil{X}}=\fil{X}$. If $X$ is connected then $\fil{X}$ is connected.\footnote{Indeed, the claim is trivial if $X$ is empty, otherwise it follows from $\fil{X}=\ds X\cup \left(\bigcup_{U} X\cup\ov{U}\right)$ where $U$ varies over the bounded connected components of $\C\setminus X$, noting that the collection of compact connected sets formed by $X$ and all the $X\cup \ov{U}$ has at least one point in common.}
\begin{definition}We will say that the bounded components of $\C\setminus X$ are \emph{shielded} by~$X$. 
\end{definition}

\begin{notation}
\label{n:hat}For a bounded set $A\subset C$, let
$$\realm{A}:=\fil{\ov{A}}.$$
\end{notation}
The following lemma is elementary:
\begin{lemma}\label{lem:inclusion}
 If $A$ and $B$ are compact subsets of $\C$ then
\[ A\subset \fil{B} \iff \fil{A}\subset\fil{B}.\]
\end{lemma}
\begin{proof}\ \\
$\Longleftarrow$: immediate since $A\subset \fil{A}$.
\\
$\Longrightarrow$: when $A\subset \fil{B}$, the unbounded component of $\C\setminus B$ is disjoint from $A$ thus it is an unbounded connected subset of $\C\setminus A$ thus it is contained in the unbounded component of $\C\setminus A$, i.e.\ $\fil{A}\subset\fil{B}$.
\end{proof}
\begin{figure}
\begin{tikzpicture}
\node at (0,0) {\includegraphics[height=266pt]{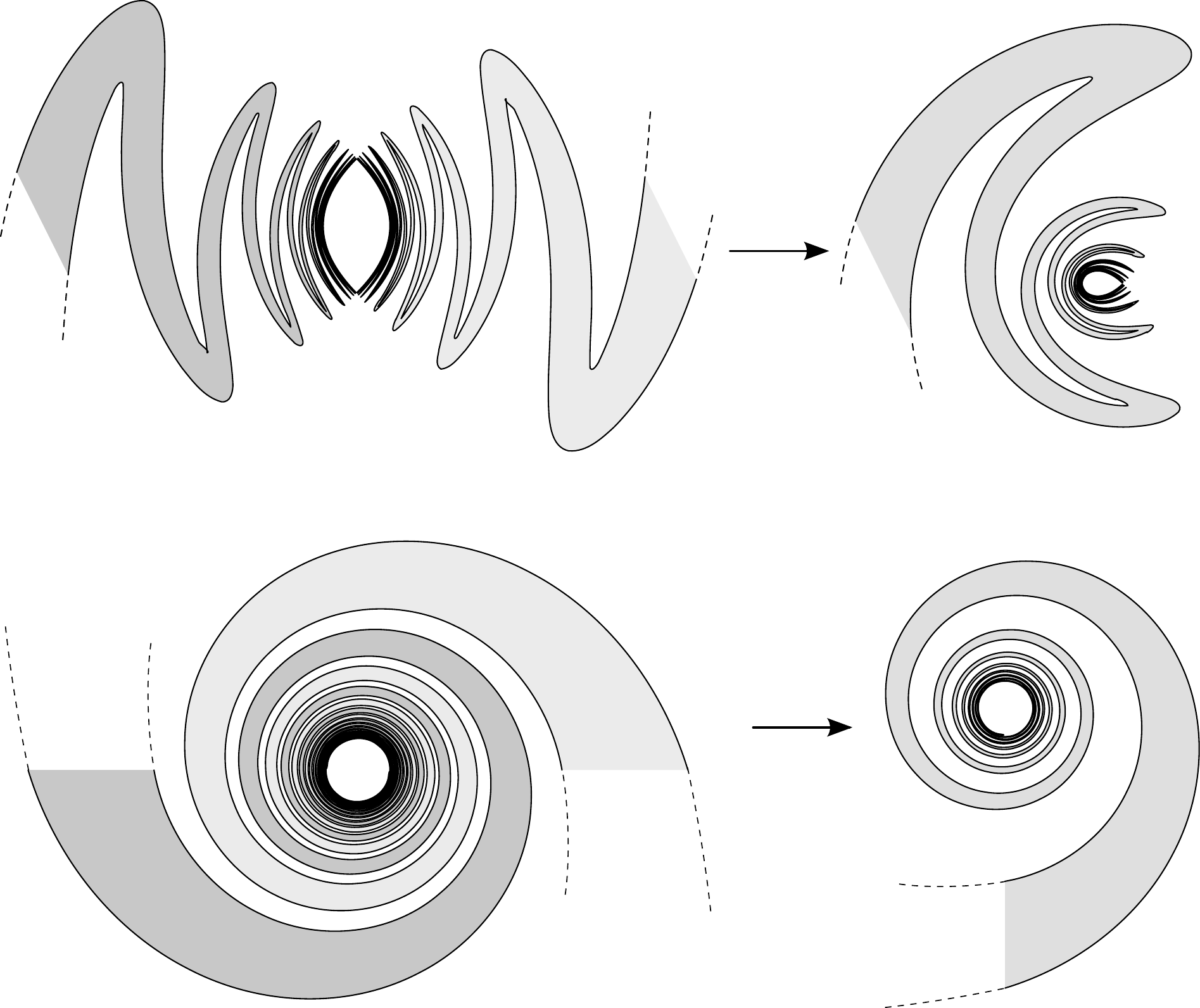}};
\end{tikzpicture}
\caption{Examples of simply connected open sets $U$ (in dark gray) and their image $V$ by $Q: z\mapsto z^2$, on which $Q$ is injective, and such that the connected component of $Q^{-1}(\realm{V})$ containing $U$ is strictly bigger than $\realm{U}$, where $\realm{V}=\fil{\ov{V}}$ and $\realm{U}=\fil{\ov{U}}$. In the top one the critical point is not in $\realm{U}$, in the bottom one it is in a shielded component (i.e.\ a bounded component of $\C\setminus\partial U$) different from $U$. Figure~\ref{fig:b} gives another example without shielded components. The situation for those hypothetical Siegel disks which are not Jordan domains might a priori include similar or more complicated features.}
\label{fig:cases}
\end{figure}

\begin{corollary}\label{cor:rincl}
 If $A$ and $B$ are bounded subsets of $\C$ then
\[ A\subset \realm{B} \iff \ov{A}\subset \realm{B} \iff \realm{A}\subset\realm{B}.\]
\end{corollary}
The first equivalence holds because $\realm{B}$ is closed, and the second follows by applying Lemma~\ref{lem:inclusion} to $\ov{A}$ and $\ov{B}$.
\begin{lemma}\label{lem:fpi}
For every complex polynomial $P$ of degree $\geq 1$ and every \emph{compact}
 subset $K$ of $\C$, $P^{-1}(\fil{K})=\fil{P^{-1}(K)}$.
\end{lemma}
\begin{proof} Let $U$ be the unbounded component of $\C\setminus K$, so that $\fil{K} = \C\setminus U$. Let $V$ be the unbounded component of $\C\setminus P^{-1}(K)$. The lemma is equivalent to $V=P^{-1}(U)$. As $V$ is connected, $P(V)$ is connected, disjoint from $K$ and unbounded thus $P(V)\subset U$, i.e.\ $V\subset P^{-1}(U)$.
Since $P$ is a finite degree ramified covering and $U$ is open, its restriction to $P^{-1}(U)$ is a finite degree ramified covering over $U$. Each component of $P^{-1}(U)$ will map surjectively to $U$, hence is unbounded, so there is exactly one component of $P^{-1}(U)$. Thus $P^{-1}(U)$ is unbounded and connected and it does not intersect $P^{-1}(K)$ (as $U$ does not intersect $K$). It follows that $P^{-1}(U)\subset V$.
\end{proof}

\begin{corollary}\label{cor:preim}
 For every complex polynomial $P$ of degree $\geq 1$ and every bounded subset $X$ of $\C$, $P^{-1}(\realm{X})=\realm{P^{-1}(X)}$.
\end{corollary}
\begin{proof} By definition, $\realm{P^{-1}(X)} = \fil{\ov{P^{-1}(X)}}$. Next, $\ov{P^{-1}(X)} = P^{-1}(\ov{X})$ (here the inclusion $\subset$ holds because $P$ is continuous, while $\supset$ holds because $P$ is an open map). Now apply Lemma~\ref{lem:fpi} with $K=\ov{X}$.
\end{proof}

The following two lemmas aim at proving Proposition~\ref{prop:chapo} below.
\begin{lemma}\label{lem:partial}
 If $X,Y\subset \C$, $X$ is compact, $Y$ is bounded and $\partial Y\subset X$ then $\ov{Y}\subset\fil{X}$.
\end{lemma}
\begin{proof}Since $\fil{X}$ is closed, it suffices to show that $Y \subset \fil{X}$. Let $A$ be the unbounded component of $\C\setminus X$. If $Y \not\subset \fil{X}$ then $Y\cap A \neq\emptyset$. As $Y$ is bounded, $A\not\subset Y$ and as $A$ is connected there exists $z\in A$ in the boundary of $Y\cap A$ relative to $A$. This point is also in the boundary of $Y$ relative to $\C$, hence in $X$ by the assumption. Thus $z\in X\cap A=\emptyset$, which is a contradiction. 
\end{proof}

\begin{lemma}\label{lem:fullcompact}
The connected components of a full compact set are full.
\end{lemma}
\begin{proof} Let $X$ be a full compact set, so $\fil{X}=X$. Let $C$ be a component of $X$. Assume by way of contradiction that there exists a bounded component $U$ of $\C\setminus C$. As $C$ is closed, $\partial U\subset C \subset X$. Hence $U$ is a bounded open set whose boundary is contained in $X$. By Lemma~\ref{lem:partial}, $\ov{U}\subset \fil{X} =X$. Now $\ov{U}\cup C$ is a connected subset of $X$, contradicting the definition of $C$.
\end{proof}

\begin{proposition}\label{prop:chapo}
Let $X$ be a compact subset of $\C$. For every connected component $C$ of $\fil{X}$,
\[C=\fil{X\cap C}.
\]
\end{proposition}
\begin{proof}
Let us prove the two inclusions $C\supset\fil{X\cap C}$ and $C\subset\fil{X\cap C}$.
\\
{$\supset$:} $X\cap C \subset C\subset\fil{C}$ so by Lemma~\ref{lem:inclusion}, $\fil{X\cap C}\subset \fil{C}$. The set $\fil{X}$ is full so by Lemma~\ref{lem:fullcompact}, $\fil{C}=C$.
\\
{$\subset$:} Let $z\in C$. If $z\in X$, then $z\in X\cap C\subset\fil{X\cap C}$. Therefore we focus on the case where 
 $z\notin X$. Since $z\in C\subset \fil{X}$ it follows that $z$ belongs to a bounded component $A$ of $\C\setminus X$.
Recall that $C$ is the connected component of $\fil{X}$ that contains $z$. Now $A$ is connected, contains $z$ and is contained in $\fil{X}$, hence $A\subset C$. Since $C$ is closed, $\partial A \subset C$, so $\partial A\subset X\cap C$. Moreover $A \cap (X\cap C) = \emptyset$. Hence $A$ is a connected component of the complement of $X\cap C$. Since $A$ is bounded, this proves $A\subset \fil{X\cap C}$.
\end{proof}

\subsection{Backward images}
The content of this section is well known; the proofs are included here for completeness. Let $P$ be a polynomial of degree $\ge 2$.

\begin{proposition}\label{prop:confluent:1}
Let $K$ be a full and non-empty compact subset of $\C$. Let $L$ be any connected component of $P^{-1}(K)$. Then there exist Jordan domains $U$, $V$ containing respectively $L$ and $K$ such that the restriction $P:U\to V$ is a ramified covering, satisfies $P^{-1}(K)\cap U=L$ and such that there is no critical value of $P$ in $V \setminus K$.
\end{proposition}
\begin{proof} 
The set $K$ has a basis of neighborhoods in $\C$ that are Jordan domains; this can be proven for instance by using a conformal map $\phi$ from $\C\setminus K$ to $\C\setminus\ov{\D}$ when $K$ is not a single point, and taking the neighborhoods $\C\setminus\phi^{-1}(\C\setminus B(0,1+\epsilon))$ which, having a connected complement, are simply connected. As there are only finitely many critical values, we can take such a neighborhood $V$ of $K$ without any in $V\setminus K$. Let $U$ be the connected component of $P^{-1}(V)$ containing $L$. Then $U$ is simply connected.
Its boundary is a Jordan curve, near which points on one side map to $V\setminus K$, and points on the other side map to $\C\setminus\ov V$. Moreover $P^{-1}(K)$ is disjoint from $\partial U$.
From the last two assertions, it follows that only one connected component of $U\setminus P^{-1}(K)$ has a closure that meets $\partial U$.
Now the map $P:U \to V$ is a ramified covering, whose restriction to $U\setminus P^{-1}( K)$ is an unramified finite degree covering of $V\setminus K$. Therefore, since $V\setminus K$ is a topological annulus (cf.\ the isomorphism $\phi$), the same must be true for each connected component of $U\setminus P^{-1}(K)$. Now in fact there is only one such connected component: indeed each of them surjects to $V\setminus K$ under $P$ thus has points very close to the Jordan curve $\partial U$.
This implies that $U\cap P^{-1}( K)$ is connected and equal to $L$.
\end{proof}

\begin{figure}
\begin{tikzpicture}
\node at (0,0) {\includegraphics[height=150pt]{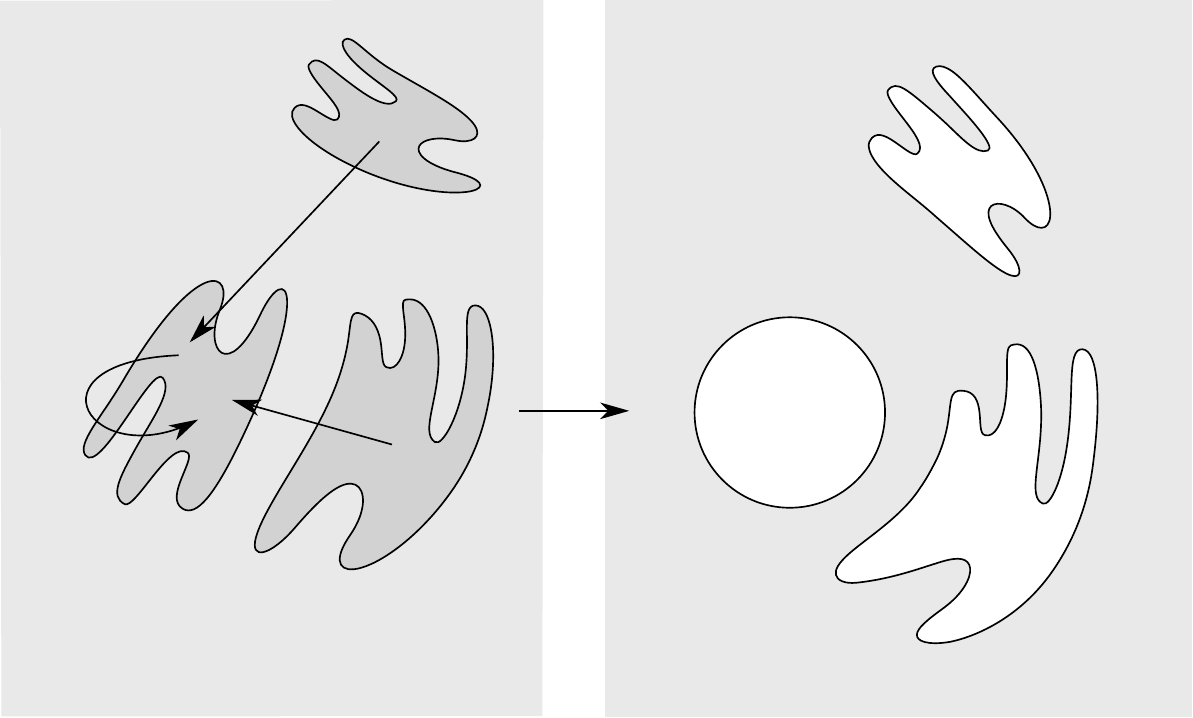}};
\node at (-3.5,0.25) {$f$};
\node at (-0.15,-0.05) {$\phi$};
\draw[->] (4.5,2.2) to[out=20,in=70,looseness=3] node[above] {$g$} (4,2.7);
\end{tikzpicture}
\caption{Schematic illustration of the construction of the external map. Imagine a polynomial fixing a Jordan domain whose preimage has $3$ disjoint connected components, including itself.
In the example shown above, they are at some distance from the Jordan domain, and the external map is well defined.
See Figures~\ref{fig:c4b} and~\ref{fig:dragons} for cases where this fails.
The domain of definition of the induced map $g$ on the right is the complement of the white set.}
\label{fig:c4}
\end{figure}
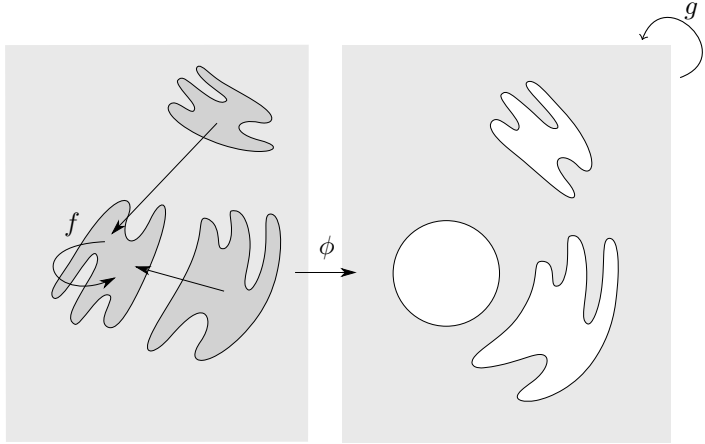

\begin{lemma}\label{lem:l}
Let $K$ and $L$ be as in Proposition~\ref{prop:confluent:1}. 
\begin{enumerate}
\item\label{item:l:1} If $L$ contains no critical point then $P$ is a homeomorphism from $U$ to $V$ carrying $L$ onto $K$.
\item\label{item:l:2} If $P:L\to K$ has a single critical value $v$ then $P^{-1}(v)\cap L$ is a singleton (thus $L$ contains a unique critical point).
\end{enumerate}
\end{lemma}
\begin{proof} Consider the (possibly) ramified covering $P: U\to V$.

\eqref{item:l:1}: A ramified covering without critical point is a covering. A covering map of a simply connected set is a homeomorphism. The rest follows.

\eqref{item:l:2}: A ramified covering with a unique critical value over a simply connected domain in the plane is topologically equivalent to $z\mapsto z^d$ from $\D$ to $\D$. The rest follows.
%
\end{proof}

\subsection{External map}\label{subsec:extmap}

\begin{figure}
\begin{tikzpicture}
\node at (0,0) {\includegraphics[height=150pt]{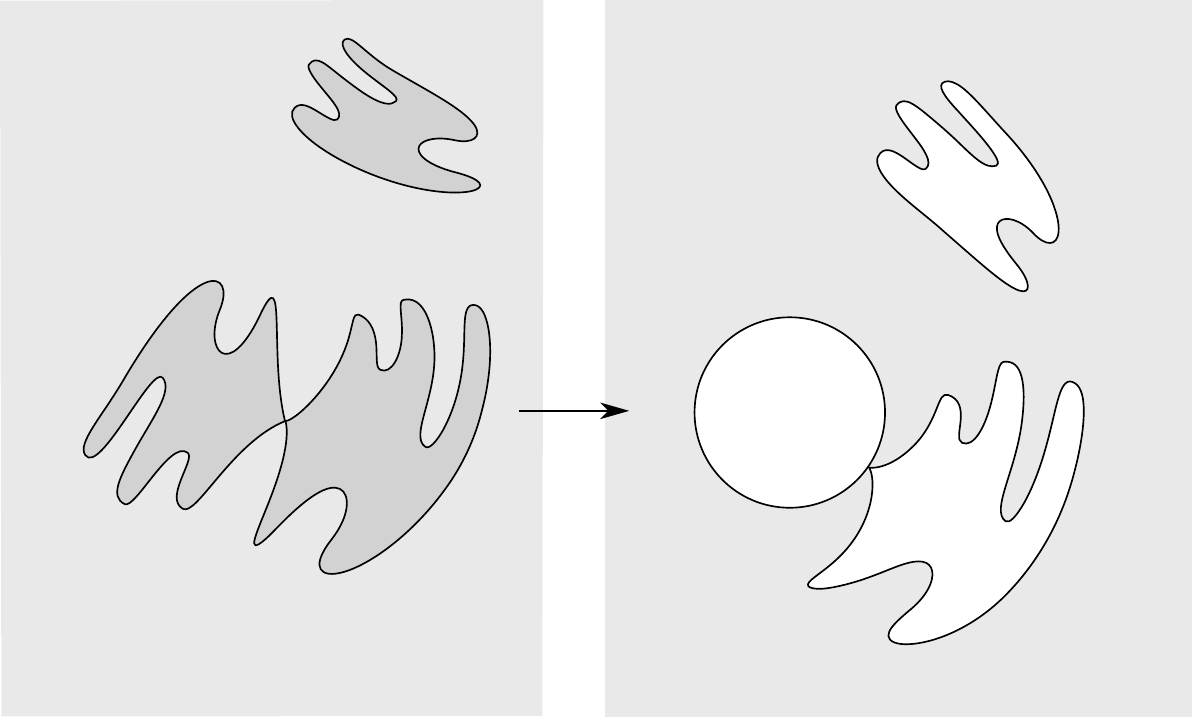}};
\node at (-0.15,-0.05) {$\phi$};
\draw[->] (4.5,2.2) to[out=20,in=70,looseness=3] node[above] {$g$} (4,2.7);
\end{tikzpicture}
\caption{Variation of Figure~\ref{fig:c4}. 
In this example, there is a critical point on the boundary and one of the preimage components touches the domain.
The external map is undefined at one point of the circle. Here the map still has a continuous extension, usually not analytic. However, one could imagine worse situations, like Figure~\ref{fig:dragons}.}
\label{fig:c4b}
\end{figure}

\begin{figure}
\begin{tikzpicture}
\node at (0,0) {\includegraphics[height=400pt]{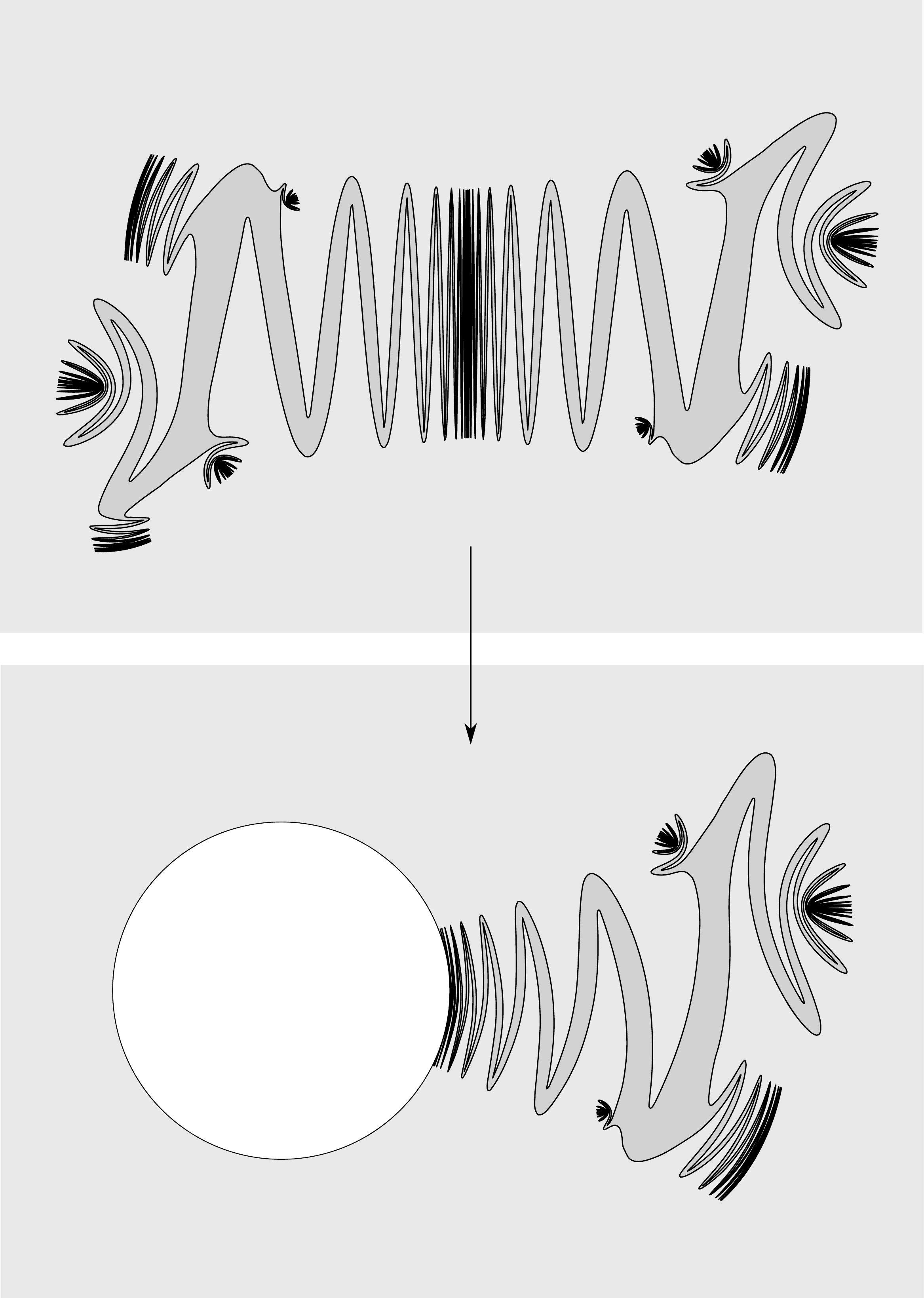}};
\node at (-0.15,0) {$\phi$};
\end{tikzpicture}
\caption{Example of how an external map could fail to be defined on a segment.}
\label{fig:dragons}
\end{figure}
This is the central tool of Ghys and Herman's proof. It has also been extensively used by Pérez-Marco in his work on hedgehogs, and by Douady and Hubbard in their work on polynomial-like maps.\footnote{They consider conjugacy classes of external maps, also known as the external \emph{classes}.} We present a particular case of the construction below, but there exist constructions in more general situations.

Consider a holomorphic map $f$ defined in an open neighborhood $\on{Dom}(f)$ of a compact subset $X$ of $\C$ and assume that $X$ is connected, contains more than one point and that $X$ is full, i.e.\ has connected complement $U=\C\setminus X$ (see also Section~\ref{sec:fill}). Assume moreover that $f(X) \subset X$ (in particular $X\subset f^{-1}(X)$) and that $X$ is \emph{locally backward invariant}, i.e.\ there is a neighborhood $U'$ of $X$ such that $f^{-1}(X) \cap U' = X$.

From the hypotheses that $X$ is non empty, compact, connected and has connected complement, we conclude that $U\cup\{\infty\}$ is a simply connected open subset of the Riemann sphere.
Let $\D\subset\C$ be the unit disk, $S^1=\partial \D$ and $V=\C\setminus\ov{\D}$.
Since $X$ has more than one point, by the Riemann mapping theorem,\footnote{To be more precise, let $s(z)=1/z$.
Then $s(U)\cup\{0\}$ is a connected non empty simply connected open subset of $\C$ with a non empty complement in $\C$.
The Riemann mapping theorem asserts that there exists a unique holomorphic bijection $\psi$ from $\D$ to $s(U)\cup\{0\}$ mapping $0$ to $0$ and such that $\psi'(0)>0$. Let $\phi=s^{-1}\circ \psi \circ s$.} there exists a holomorphic bijection $\phi$ from $V$ to $U$, with the property that $\phi(z)/z$ tends to a positive real as $z\tend \infty$. Moreover, since $\phi$ is a homeomorphism, for every sequence $(z_n)$ in $V$, $z_n\tend \partial V$ if and only if $\phi(z_n)\tend\partial U$.
Let us conjugate $f$ by $\phi$ by setting $g = \phi^{-1} \circ f \circ \phi$. The domain of definition of $g$ is the set of points in $V$ whose image by $\phi$, then by $f$, do not belong to $X$.
This domain contains the annulus $1<|z|<1+\epsilon$ for some $\epsilon>0$ because of the local backward invariance.
Also, for every sequence $(z_n)$ in $\on{Dom}(g)$, $z_n\tend S^1$ implies $g(z_n)\tend S^1$.
Therefore, the map $g$ admits a Schwarz reflection extension $\wt{g}$ to the following domain: the union of $S^1$, $\on{Dom}(g)$ and the reflection of $\on{Dom}(g)$ with respect to $S^1$.
This domain contains the annulus $1/(1+\varepsilon) < |z| < 1+\varepsilon$.
The map $\wt{g}$ is holomorphic, preserves the circle $S^1$ and commutes with the reflection $z \mapsto 1/\bar{z}$.
Its restriction to $S^1$ is called the \emph{external map} and the map $\wt{g}$ will be called the \emph{extended external map} associated with the pair $(f,X)$.
See Figures~\ref{fig:c4}, \ref{fig:c4b} and~\ref{fig:dragons}.
The external map is a non-constant orientation-preserving analytic self-map of the circle without critical points, in particular it is a covering.
This follows from the fact that $\wt{g}$ maps all points outside the circle to points outside the circle. 

\subsection{The separation theorem}\label{subsec:sep}
We will use the theory of polynomial-like maps which can be found in~\cite{aap}.

The following theorem is stated in~\cite{Kiwi} as Lemma 3.1 and as an immediate corollary of a theorem of Goldberg and Milnor.

\begin{theorem*}[Goldberg, Milnor, Poirier, Kiwi]
Let $P$ be a complex polynomial of degree $\geq 2$ and connected Julia set. Then
there exists $m>0$ such that the union of the set of external rays fixed by $P^m$ (there are finitely many of them) and their landing points cut the plane into regions, each of which contains at most one periodic Fatou component or Cremer point (but never both).
\end{theorem*}

The external rays in the above theorem are periodic under $P$. 
As such, they land, and their landing points are either repelling or parabolic periodic points. This implies:

\begin{corollary}\label{cor:sep}
Let $P$ be a degree $>1$ complex polynomial with connected \emph{or disconnected} Julia set. Let $U, V$ be distinct \emph{periodic} bounded 
Fatou components of $P$, and $\realm{U}=\fil{\ov{U}}$, $\realm{V}=\fil{\ov{V}}$. Then $\realm{U}\cap \realm{V}$ is either empty or reduced to a periodic point. In the latter case, this point is either parabolic or repelling.
\end{corollary}
\begin{proof}
This is an immediate consequence of the separation theorem in the case the Julia set $J(P)$ is connected. Otherwise, we use the following observation: let $C$ be a connected component of the filled-in Julia set $K(P)$. If $P^p(C)\cap C\neq \emptyset$ then $P^p(C)=C$ and there exists a polynomial-like restriction of $P^p$ whose filled-in Julia set is $C$ (see Lemma~\ref{lem:bh} below).
Now we choose for $C$ the component of $K(P)$ that contains $U$. Then $\realm{U}\subset C$. If $V$ is contained in a component of $K(P)$ other than $C$, then $\realm{U}\cap \realm{V}=\emptyset$. Otherwise consider an associated polynomial-like restriction of $P^p$ as above, where $p$ is the period of $C$. This restriction has degree $>1$ since otherwise $C$ would reduce to a point. Therefore it is hybrid-equivalent to a polynomial of degree $\geq 2$ map near their Julia sets (see \cite{aap}). So we can reduce the argument to the connected case as above.
\end{proof}

The following result was used in the above proof. For completeness we include its proof, borrowed from \cite{BH2}.

\begin{lemma}\label{lem:bh}
Assume $P$ is a degree $d\geq 2$ polynomial, $C$ is a connected component of the filled-in Julia set $K(P)$, and $P^p(C)\cap C\neq \emptyset$ for some $p\geq 1$. Then $P^p(C)=C$ and there exists a polynomial-like restriction of $P^p$ whose filled-in Julia set is $C$.
\end{lemma}
\begin{proof}
The set $P^p(C)$ is a connected subset of $K(P)$, so $P^p(C)\subset C$. 
The reverse inclusion follows from the \emph{confluence} property (see \cite{Why}): if $f$ is any open map of the Riemann sphere, $L$ a connected compact subset of the Riemann sphere, and $K$ a connected component of $f^{-1}(L)$ then $f$ maps $K$ onto $L$. Let $f=P^p$, $L=C$ and $K=C$. Since $K$ is a connected component of $K(P)$ contained in $f^{-1}(L)$, it is a fortiori a connected component of $f^{-1}(L)$.

Let $G: \C \to [0,+\infty[$ be the Green's potential associated with $K(P)$.
Given $x>0$, let $U_x$ be the connected component
of the set $G^{-1}([0,x[)$ containing $C$. 
Then $U_x$ is a simply connected open set.
Also, $\bigcap_{x>0} \ov{U}_x$ is connected and contained in $K(P)=G^{-1}(0)$.
It follows that this intersection is equal to $C$. In particular, for every neighborhood $V$ of $C$ in $\C$, there is some $x$ such that $U_x\subset V$.
Take now $x$ small enough so that there are no critical points of $P$ in $U_x\setminus C$. Then the filled-in Julia set of the polynomial-like restriction $P^p: U_x \to U_{d^p x}$ is connected, contained in $K(P)$ and contains $C$, so it is equal to $C$.\end{proof}

\section{Dynamics of Filled Siegel Disks}
\subsection{Setting up}\label{subsec:setup}Denote by $P$ a polynomial of degree $\geq 2$ with a Siegel disk $\Delta$. 
Let (see Section~\ref{sec:state} and Section~\ref{sec:fill} notation~\ref{n:hat})
 \[\realm{\Delta}=\fil{\ov{\Delta}}.\]
Thus, $\realm{\Delta}$ is the disjoint union of $\Delta$, $\partial\Delta$ and at most countably many other bounded components of $\C\setminus\partial\Delta$, if any.\footnote{As we remarked earlier, there are examples of holomorphic maps defined on a simply connected open subset $U$ of $\C$ with a Siegel disk compactly contained in $U$ and whose boundary is not locally connected (for instance a pseudocircle, see \cite{C1}). Yet, no such examples are known with a Siegel disk having a boundary with more than two complementary components. Like chimaeras, they may not exist.}

Recall that we are not only considering fixed Siegel disks but also Siegel disks with a higher period. Let $p$ be the period of the Siegel disk $\Delta$.

For $k\geq 0$, let
\[ \Delta^k=P^k(\Delta)
\]
so the orbit of $\Delta$ is $\Delta=\Delta^0\mapsto\Delta^1\mapsto\cdots\mapsto\Delta^{p-1}\mapsto\Delta^p=\Delta$.

Let
\[\realm{\Delta}^k = \fil{\ov{\Delta^k}}.
\]
\begin{remark}
We used the notation $\realm{\Delta}^k$ instead of $\realm{\Delta^k}$ to avoid extra wide hats like in $\realm{\Delta^{k+1}}$. However, the set $\fil{\ov{P^k(\Delta)}}$ is a priori not equal to $P^k(\fil{\ov{\Delta}})$, so attention should be paid to the fact that $\realm{\Delta}^k$ does \emph{not} denote the latter.
\end{remark}

Applied to our situation, Corollary~\ref{cor:sep} (separation theory) implies:
\begin{corollary}\label{cor:sep:2}\ 
\begin{itemize}
\item For $l\neq k \pmod{p}$, $\realm{\Delta}^k \cap \realm{\Delta}^l$ is either empty or a single point $z$, in which case $z$ is a repelling or parabolic periodic point on which at least two external rays $\gamma$, $\gamma'$ land. Moreover, $\gamma\cup\{z\}\cup\gamma'$  separates $\realm{\Delta}^k\setminus\{z\}$ from $\realm{\Delta}^l\setminus\{z\}$.
\item A critical point can belong to at most one $\realm{\Delta}^k$.
\end{itemize}
\end{corollary}

\subsection{Dynamics}\label{subsec:dyn}
Let $A$ denote the basin of attraction of infinity for $P$. Recall that the Julia set $J$ is equal to the boundary $\partial A$ and that the Fatou set is the complement of $J$. Siegel disks of rational maps are components of the Fatou set, thus
\[ \partial \Delta^k \subset J.\]

Recall that for a rational map $F$ and a Fatou component $U$ of $F$, $F(U)$ is a Fatou component of $F$ and $F(\partial U) = \partial F(U)$.
Hence
$$P(\partial \Delta^k)=\partial \Delta^{k+1}.$$

\begin{lemma} \label{l:Fatou} \ 
\begin{itemize}
\item Any non-empty bounded and connected open subset of $\C$ whose boundary is contained in $J$ is a Fatou component.\footnote{Note that this is false if we do not assume boundedness: consider for instance the complement of the closure of any bounded Fatou component of Douady's rabbit.}
\item Bounded components of $\C\setminus \partial \Delta^k$ are Fatou components.
\end{itemize}
\end{lemma}
\begin{proof}
 First claim: call this set $U$. If $U$ has non empty intersection with a Fatou component $V$ then $V\subset U$: indeed $U$ is open and $\partial U\subset J$ therefore $U\cap V$ is both open and closed in $V$.
Now $U$ cannot contain points in $J=\partial A$ (where $A$ is the basin of infinity), for otherwise as an open set it would contain points in $A$, and therefore it would contain $A$ and be unbounded. So $U$ is contained in a Fatou component,\footnote{Alternatively, $\partial U\subset J$ which is bounded and invariant; hence the iterates of $f$ are bounded on $\partial U$ thus on $U$ by the maximum principle; whence $U$ is contained in the Fatou set.} and thus is equal to it.
\\
 Second claim: follows from the first and the inclusion $\partial \Delta^k\subset J$.
\end{proof}


\begin{lemma}\label{lem:oubli}
For every $k$, $\partial\realm{\Delta}^k = \partial\Delta^k$.
\end{lemma}
\begin{proof} Without loss of generality, we prove the case of $\Delta^0=\Delta$.
For any compact subset $X$ of $\mathbb C$, $\partial \fil{X} \subset \partial X$, hence $\partial\realm{\Delta} = \partial \fil{\ov{\Delta}} \subset \partial \ov\Delta\subset\partial \Delta$. The other inclusion is more specific to $\Delta$: $\partial \Delta\subset J$ hence $\partial\Delta\subset\ov{A}$.
Together with $\partial \Delta\subset\realm{\Delta}$ and $A\cap\realm{\Delta} =\emptyset$, this implies $\partial\Delta \subset\partial\realm{\Delta}$.
\end{proof}

\begin{lemma}\label{lem:dsp}
The image of a bounded component of $\C\setminus \partial\Delta^k$ is a bounded component of $\C\setminus \partial\Delta^{k+1}$.
In particular,
\[ P\left(\realm{\Delta}^k\right)\subset\realm{\Delta}^{k+1}.
\]
\end{lemma}
\begin{proof}
Consider a bounded component $U$ of $\C\setminus\partial\Delta^{k}$ so $\partial U\subset \partial \Delta^k$. We have\footnote{If $f:X\to Y$ is a continuous and open map between topological spaces and $U$ is an open subset of $X$ with compact closure, then $\partial f(U)\subset f(\partial U)$.} $\partial P(U)\subset P(\partial U)$, and as already mentioned $P(\partial \Delta^{k})=\partial \Delta^{k+1}$, thus $\partial P(U)\subset \partial \Delta^{k+1}$.
By Lemma~\ref{l:Fatou}, $U$, thus $P(U)$, are contained in the Fatou set. In particular, $P(U)\cap \partial \Delta^{k+1} = \emptyset$. Now $P(U)$ is an open connected subset of the complement of $\partial \Delta^{k+1}$ whose boundary is contained in $\partial \Delta^{k+1}$, hence it is a connected component of $\C\setminus \partial \Delta^{k+1}$.
\end{proof}

\subsection{Backward images toward local total invariance}\label{subsec:aprop}
The sets that we are about to define will be used here and later in the text.

 \begin{definition}\label{d:apo}
Let $\apo{\Delta}^k$ be the connected component of $P^{-1}(\realm{\Delta}^{k+1})$ that contains $\realm{\Delta}^{k}$.
\end{definition}

\begin{definition}[The sets $U_k$ and $V_k$]\label{def:uv}
Applying Proposition~\ref{prop:confluent:1} to $K=\realm{\Delta}^k$, there are simply connected open neighborhoods $U_{k-1}$, $V_k$ of $\apo{\Delta}^{k-1}, \realm{\Delta}^k$ respectively, for $k\in \{1,\cdots, p\}$, such that the restriction $P\vert_{U_{k-1}} : U_{k-1}\to V_{k}$ is a ramified covering satisfying $P^{-1}( \realm \Delta^{k}) \cap U_{k-1}=\apo\Delta^{k-1}$ and such that there is no critical value of $P$ in $V_k \setminus\realm{\Delta}^{k}$. Let us define $U_{k+p}=U_k$ and $V_{k+p}=V_k$ for all $k\in\Z$. Note that there is no a priori inclusion between $U_k$ and $V_k$.
\end{definition}

\begin{corollary}\label{cor:zero}
If $\apo{\Delta}^k$ does not contain a critical point then $\apo{\Delta}^k=\realm{\Delta}^k$ i.e.\ $\realm{\Delta}^k$ is a connected component of $P^{-1}(\realm{\Delta}^{k+1})$. Moreover, $P$ is a bijection from $\realm{\Delta}^k$ to $\realm{\Delta}^{k+1}$.
\end{corollary}
\begin{proof}By Lemma~\ref{lem:l} case~\eqref{item:l:1}, $P$ is a bijection from $\apo{\Delta}^k$ to $\realm{\Delta}^{k+1}$. In particular $P$ is injective on $\partial \Delta^k$. Recall that $P(\partial \Delta^k)=\partial \Delta^{k+1}$.
Hence \[P^{-1}(\partial \Delta^{k+1})\cap \apo{\Delta}^k = \partial \Delta^k.\]
Now consider a point $z\in\apo{\Delta}^k$. Then $P(z)\in\realm{\Delta}^{k+1}$. Either $P(z)\in\partial \Delta^{k+1}$ or $P(z)$ belongs to a bounded component $W$ of $\C\setminus \partial \Delta^{k+1}$.
In the first case, $z\in\partial \Delta^k$ thus $z\in\realm\Delta^k$.
In the second case, consider the component $W'$ of $P^{-1}(W)$ that contains $z$. Since $P(\partial W')\subset\partial W\subset \partial\Delta^{k+1}$, we get $\partial W'\subset\partial \Delta^k$. Hence $W'$ is a bounded connected component of $\C\setminus\partial\Delta^k$: $W'\subset \realm{\Delta}^k$.
\end{proof}

\begin{corollary}\label{lem:locbwinv}
If $\apo{\Delta}^k=\realm{\Delta}^{k}$ for every $k$, then $\realm{\Delta}$ is \emph{locally backward invariant} under $P^p$, i.e.\ 
there is a neighborhood $W$ of $\widehat{\Delta}$ such that $P^{-p}(\widehat{\Delta}) \cap W = \widehat{\Delta}$.
\end{corollary}
\begin{proof}
By the assumption, for each $k$ the restriction $P: U_{k-1} \to V_k$ pulls $\realm{\Delta}^k$ back to $\realm{\Delta}^{k-1}$. It is then not hard to check that the open set
\[ W=\bigcap_{k=0}^{p-1} P^{-k}(U_k)
\]
has the required property.
\end{proof}

The converse also holds (but will not be used in this article). Note that under the assumptions of the last statement, $P^p(\realm{\Delta}) = \realm{\Delta}$ so $\realm{\Delta}$ is in fact locally totally invariant.

\subsection{Shielded components eventually map to the Siegel disk}\label{subsec:unl}

Recall that a component shielded by $\partial \Delta^k$ is a Fatou component and is mapped to a component shielded by $\partial \Delta^{k+1}$ (Lemma~\ref{lem:dsp}). The following statement is proved in~\cite{Rogers}. We reproduce its proof here for convenience.

\begin{lemma}[Rogers, Theorem 3.3 in \cite{Rogers}]\label{lem:rog}
 All components shielded by $\partial \Delta^k$ eventually map under the iteration of $P$ to some (hence to every) $\Delta^l$.
\end{lemma}
\begin{proof}
 Let $U$ be a component shielded by $\partial \Delta^k$. By Lemma~\ref{l:Fatou}, $U$ is a Fatou component. By Sullivan's no-wandering-domain theorem, it is eventually mapped to a periodic component $V=P^n(U)$.
If $V$ were none of the $\Delta^l$, then by Corollary~\ref{cor:sep} (consequence of the separation theorem)
$\partial V$ would intersect each $\partial \Delta^l$ in at most one point. 
However, $\partial V\subset \bigcup_k\partial\Delta^k$ and $\partial V$ is infinite. Contradiction.
\end{proof}

\begin{complem}\footnote{Rogers proved a stronger statement in \cite{Rogers} Theorem 8.4: if $P$ is injective on the boundary of a Siegel disk of period one, then $\Delta$ is the only component shielded under $\partial \Delta$. It implies our Special Case because here $P^p$ is injective on $\partial \Delta^k$, see Corollary~\ref{cor:zero}. 
}\label{complement} 
 If there is no critical point in $\bigcup_{k\geq 0} \apo{\Delta}^k$ then for every $m$, $\Delta^m$ is the only component shielded by $\partial \Delta^m$.
\end{complem}
\begin{proof} 
By Corollary~\ref{cor:zero} the map $P$ is a bijection from $\realm{\Delta}^k$ to $\realm{\Delta}^{k+1}$ for every $k$. Thus $\Delta^{k+1}$ has only one preimage by $P$ in $\realm{\Delta}^k$, namely $\Delta^k$. Any component $V$ shielded by $\partial \Delta^m$ has its orbit contained in $\bigcup_k\realm{\Delta}^k$, and eventually maps to some $\Delta^l$ by Lemma~\ref{lem:rog}. Thus $V=\Delta^m$.\end{proof}
Note that this result makes no assumption on the number of critical values of $P$.

\subsection{Proof of Corollary~\ref{cor:p} (Herman's theorem in the periodic case)}\label{sub:sup1}

For convenience, we recall its statement:
\emph{
\begin{enumerate}
\item For every rational map $f$ of degree $\geq 2$ having a Siegel disk $\Delta$ of period $p$ with rotation number in $\cal H$, there exists an $i$ between $0$ and $p-1$ such that $f$ is not injective in any neighborhood of $\partial f^i(\Delta)$.
\item For every unicritical polynomial $f(z)=z^d+c$ having a Siegel disk $\Delta$ of period $p$ and with rotation number in $\cal H$, there exists an $i$ between $0$ and $p-1$ such that $f$ has a critical point on $\partial f^i(\Delta)$.
\end{enumerate}%
}%
\noindent Let us also recall the following equivalence (stated on page~\pageref{rem:one} as a remark just after Herman's theorem), where $K$ is a compact subset of the domain of definition of a holomorphic map $f$.

\emph{The following are equivalent:
\begin{enumerate}[label=(\roman*)]
\item There is no neighborhood of $K$ on which $f$ is injective.
\item Either $f$ has a critical point on $K$ or the restriction of $f$ to $K$ is non-injective (or both).
\end{enumerate}
} 

Assume $p>1$, since the case $p=1$ is already covered by Herman's theorem. Let $\Delta^i=f^i(\Delta)$. 
As recalled in Section~\ref{subsec:dyn}, $f^i(\partial \Delta) =\partial \Delta^i$. 
We treat the two parts of the theorem separately:

Part~\ref{item:cor:p:1} ($f$ is a rational map). Herman's theorem applied to $f^p$ shows that $f^p$ is not injective in any neighborhood of $\partial\Delta$. By the remark above, this implies that
either there is a critical point of $f^p$ on $\partial \Delta$ or that $f^p$ is non-injective on $\partial\Delta$. In the first case, $f$ has a critical point on $\Delta^i$ for some $i$ between $0$ and $p-1$. In the second case, $f$ is non-injective on $\partial \Delta^i$ for some $i$ between $0$ and $p-1$. Applying the remark again, we conclude that in either case $f$ is not injective in any neighborhood of $\partial \Delta^i$.

Part~\ref{item:cor:p:2} ($f$ is a unicritical polynomial). We give here a proof using the following lemma.
See the end of this section for an alternate proof.
\begin{lemma}[folk.]\label{lem:upl}
If $f$ is a unicritical polynomial with a period $p$ Siegel disk $\Delta$, then 
there is a unicritical polynomial-like restriction of $f^p$ whose filled-in Julia set contains $\Delta$.
\end{lemma}
\begin{proof} We give only a brief sketch.
Consider $\Delta^i=f^i(\Delta)$ and the regions cut by the union of the external rays fixed by $f^p$ and their endpoints.
By the Fatou-Shishikura theorem,\footnote{The number of non-repelling cycles of a polynomial map is at most the number of its critical points in $\C$.} $f$ has at most one non-repelling cycle, hence all these endpoints are repelling.
By the separation theorem (see Section~\ref{subsec:sep}) applied to $f^p$, each region contains at most one $\Delta^i$. Let $V$ be the region containing $\Delta$ and $U$ be the connected component of $f^{-p}(V)$ containing $\Delta$. Then $f^p$ is a proper map from $U$ to $V$ and $U\subsetneq V$. For each distinct $i,i'$ with $0\leq i <p$ and $0\leq i' <p$, the regions containing $f^i(U)$ and $f^{i'}(U)$ are disjoint, so at most one contains a critical point. As a consequence, the restriction of $f^p$ from $U$ to $V$ is univalent or unicritical. It cannot be univalent for that would violate Schwarz's lemma at the center of the Siegel disk. Now one can cut $U$ and $V$ by equipotentials and thicken external rays and endpoints to get the desired polynomial-like restriction; see~\cite{GM} for details.
\end{proof}
Polynomial-like maps are quasiconformally conjugate to polynomials in a neighborhood of their respective filled-in Julia sets \cite{aap}. In particular in our case, $f^p$ has a restriction to an open set containing $\Delta$ which is conjugate to a unicritical polynomial.
By the period one version of Herman's theorem, $f^p$ has a critical point on $\partial \Delta$, which means that there is a critical point on $\partial \Delta^i$ for some $i$.
This completes the proof of Corollary~\ref{cor:p}.

\medskip

Here is an alternate proof of Part~\ref{item:cor:p:2} that avoids the use of Lemma~\ref{lem:upl}. According to Fatou (see \cite{Milnor} Theorem 11.17), the $\omega$-limit set of the critical point $0$ contains the boundary of $\Delta^i$ for every $i$. By Sullivan's no-wandering-domain theorem and the classification of Fatou components, $0$ cannot belong to a Fatou component.
Now, if $0$ is not on the boundary of any $\Delta^i$, then according to Part~\ref{item:cor:p:1} of Corollary~\ref{cor:p} and Case~\ref{item:rem:ii} of the remark applied to $f^p$ on $\partial \Delta$, $f^p$ is not injective on $\partial \Delta^i$. Equivalently there is some $i$ such that $f$ is not injective on $\partial\Delta^i$.
By Corollary~\ref{cor:zero}, the critical point $0$ must belong to $\apo{\Delta}^i$ and the critical value $P(0)$ to $\realm{\Delta}^{i+1}$. Since $P(0)$ is not in the Fatou set, this means that $P(0)\in\partial \Delta^{i+1}$. Then since $0$ is the only preimage of $P(0)$ and since $P$ is always surjective from $\partial \Delta^i$ to $\partial \Delta^{i+1}$, this means that $0\in\partial \Delta^i$, contradicting our assumption.


\section{Reduction of the Main Theorem}\label{s:reduction}
We now describe our plan for the proof of Theorem~\ref{thm:main}.
The proofs of the three main steps stated below are the object of the last three sections.


Let $P$ be a polynomial of degree $d\ge 2$ with 
a Siegel disk $\Delta$ of arbitrary period $p$. Let $\Delta^k=P^k(\Delta)$ and $\realm{\Delta}^k=\fil{\overline{\Delta^k}}$.

\begin{lemma}\label{lem:precii}
 If all critical orbits eventually enter $\bigcup_{k\geq 0} \realm{\Delta}^k$ then there is a critical point on $\bigcup_{k\geq 0}\partial \Delta^k$.
\end{lemma}
\begin{proof} Denote by $\omega(z)$ the $\omega$-limit set of $z$, i.e.\ the set of points of accumulation of the sequence $(P^n(z))$. Let $B=\bigcup_{k\geq 0}\partial \Delta^k$.

Consider a critical point $c$ that belongs to the Fatou set of $P$. By Sullivan's no-wandering-domain theorem and Fatou's classification of periodic components, $\omega(c)$ is equal to an attracting cycle, a parabolic cycle, or a cycle of invariant curves in Siegel disks. The intersection $\omega(c)\cap J$ is therefore finite or empty.

Another theorem of Fatou asserts that $\partial\Delta\subset \bigcup_c \omega(c)$, the union being over all critical points.
The set $\partial\Delta$ is contained in $J$ and contains infinitely many points. 

Recall that $\realm{\Delta}^k$ is the disjoint union of $\partial\Delta^k$ and the bounded components of $\C\setminus\partial\Delta^k$, and that the latter are Fatou components (Lemma~\ref{l:Fatou}).
We have assumed that all critical orbits of $P$ eventually belong to some $\realm{\Delta}^k$. If none would fall on $B$, they would all belong to the Fatou set, thus $J\cap \big(\bigcup_c \omega(c)\big)$ would be finite 
contradicting the fact that it contains the infinite set $\partial \Delta$.

To prove that there is not only a point in the critical orbit, but also a critical point on $B$, we must use the following theorem of Mañé \cite{Mane1}, much harder than Fatou's: there exists a \emph{recurrent} critical point $c$ such that $\partial \Delta\subset\omega(c)$. As above, $P^n(c)$ must belong to $B$ for some $n\in\N$. Since $P(B)=B$, we have $\omega(c)=\omega(P^n(c))\subset B$. Since $c\in \omega(c)$, we have $c\in B$.
\end{proof}

Recall that $\apo{\Delta}^k$ is defined as the connected component of $P^{-1}(\realm{\Delta}^{k+1})$ containing $\realm{\Delta}^{k}$.

\begin{definition}
Let $n_0$ be the number of critical points of $P$ belonging to $\bigcup_{k\geq 0} \apo{\Delta}^k$, and $n_1$ be the number of critical values in $\bigcup_{k\geq 0} \realm{\Delta}^k$, both counted \emph{without} multiplicity.
\end{definition}

\begin{theorem} [Herman's Case]\label{thm:preci}If $n_0=0$ then $\theta$, the rotation number of $\Delta$, is not in $\cal H$.
\end{theorem}

Note that by the Special Case in Section~\ref{subsec:unl}, under the same hypothesis, $\Delta$ is the only component shielded by $\partial \Delta$. 

\begin{lemma}\label{lem:oops}
If $n_1=1$ then $n_0 = 0$ or $n_0= 1$.
\end{lemma}

\begin{theorem}\label{thm:preciii}
 If $n_0=1$ and $P$ has only two critical values then there is a critical point on $\bigcup_{k\geq 0}\partial \Delta^k$.
\end{theorem}
\vskip 1ex 
\noindent{\it Proof of our Main Theorem assuming Lemma~\ref{lem:oops}, Theorem~\ref{thm:preci} and Theorem~\ref{thm:preciii}:}

Let $P$ be a polynomial of degree $d\ge 2$ with two finite critical values, and a Siegel disk $\Delta$ of arbitrary period $p$. Necessarily $n_1\le 2$. If $n_1=2$, all critical orbits eventually enter $\bigcup_{k\geq 0} \realm{\Delta}^k$ and by Lemma~\ref{lem:precii} there is a critical point on $\bigcup_{k\geq 0}\partial \Delta^k$.
 If $n_1=1$ then $n_0 = 0$ or $n_0= 1$ by Lemma~\ref{lem:oops}.
 Last, if $n_1=0$ then $n_0=0$.
If $n_0=0$, it follows from Herman's case (Theorem~\ref{thm:preci}) that $\theta$, the rotation number of $\Delta$, is not in $\cal H$.
If $n_0=1$, then Theorem~\ref{thm:preciii}
ensures that then there is a critical point on $\bigcup_{k\geq 0}\partial \Delta^k$.
\qed

\section{Proof of Herman's Case}

 The idea of the proof is the same as Herman's in \cite{H1}, which in turns generalizes a result of Ghys in \cite{G2}.
Assume that $n_0=0$.
By Corollary~\ref{cor:zero}, $\apo{\Delta}^k=\realm{\Delta}^k$ for all $k$. By Corollary~\ref{lem:locbwinv}, 
$\realm{\Delta}$ is locally backward invariant under $P^p$ in the sense of Section~\ref{subsec:extmap}, so we can consider the external map associated with the pair $(P^p,\realm{\Delta})$. 
Recall from Section~\ref{subsec:extmap} that this external map is defined as the restriction to the unit circle of the Schwarz reflection $\wt f$ of $f:=\phi\circ P^p\circ\phi^{-1}$ for the unique conformal map $\phi:\C\setminus\realm{\Delta}\to\C\setminus\ov{\D}$ such that $\phi(z)/z$ has a positive limit at $\infty$. 

\begin{lemma}
The restriction of $\wt{f}$ to $\partial \D$ is an orientation-preserving analytic diffeomorphism with rotation number $\theta$.
\end{lemma}
\begin{proof}
By Section~\ref{subsec:extmap} we already know that $\wt{f}$ is analytic and that its restriction is orientation-preserving and has no critical point. 
To prove that $\wt{f}$ has degree one
on $\partial\D$ and has the same rotation number as the Siegel disk, Herman's trick is to do the same construction as above but replacing $\realm{\Delta}$ by the invariant sub-disk $\Delta_r:=\psi(\ov{B}(0,r))$ where $r<1$ and $\psi:\D\to\Delta$ is a conformal map sending $0$ to the center of the Siegel disk, i.e.\ to the periodic point. Recall that the map $\psi$ is also a conjugacy from the rotation of angle $2\pi\theta$ to $P^p$. Let $\phi_r$ be the unique conformal map from $\C\setminus\overline{\Delta}_r$ to $\C\setminus\ov{\D}$ such that $\phi_r(z)/z$ has a positive limit at $\infty$ and let $f_r=\phi_r\circ P^p\circ \phi_r^{-1}$. Since the boundary of $\Delta_r$ is a Jordan curve, 
the map $\phi_r$ extends to a homeomorphism from $\partial\Delta_r$ to $\partial \D$. The map $\phi_r\circ\psi$ conjugates the rotation of angle $2\pi\theta$ on the circle of radius $r$ to $f_r$ on $\partial \D$. Thus \emph{the rotation number of $f_r$ is equal to $\theta$}. 
The domain $\C\setminus\overline\Delta_r$ tends to $\C\setminus\realm{\Delta}$ in the sense of Caratheodory.\footnote{Consider $r_n\tend 1$. By definition the associated \emph{kernel} is the union over $n\in\N$ of the unbounded connected component $V_n$ of the interior of $K_n=\bigcap_{k\geq n} \C\setminus\overline\Delta_{r_k}$. Trivially, $K_n=\C\setminus\Delta$, so $V_n =\C\setminus\realm{\Delta}$. Hence the kernel is $\C\setminus\realm{\Delta}$. Thus the family $\C\setminus \Delta_r$ converges in the sense of Caratheodory to $\C\setminus\realm{\Delta}$ as $r\tend 1$.} Thus as $r\tend 1$, $\phi_r\tend \phi$ uniformly on compact subsets of $\C\setminus\realm{\Delta}$ and $\phi_r^{-1}\tend \phi^{-1}$ uniformly on compact subsets of $\C\setminus\ov{\D}$.
We claim that apart from $\Delta$ itself, the components of $P^{-p}(\Delta)$ are at positive distance from $\realm{\Delta}$. By local backward invariance, $P^{-p}(\realm{\Delta})\setminus \realm{\Delta}$ is at positive distance from $\realm{\Delta}$. Since $P$ is injective on $\realm{\Delta}^k$ by Corollary~\ref{cor:zero}, it follows that $P^{-p}(\Delta)\cap \realm{\Delta} = \Delta$, proving the claim.
Now, since $\phi_r^{-1}\tend \phi^{-1}$ uniformly on compact sets, the image by $\phi_r^{-1}$ of the annulus $1<|z|<1+\epsilon/2$ is disjoint from the components of $P^{-p}(\Delta)$ other than $\Delta$ for all $r$ sufficiently close to $1$.
Thus the domains of definition of $f_r$ contain the annulus $1<|z|<1+\epsilon/2$ for all such $r$.
Thus the domains of definition of $\wt{f}_r$ contain a common annulus containing $\partial\D$ for $r$ close enough to $1$.
On the boundary of this annulus, $\wt{f}_r$ converges uniformly to $\wt{f}$.
This implies by the maximum principle that $\wt{f}_r$ converges uniformly to $\wt{f}$ on the annulus, so in particular on $\partial \D$.
As a uniform limit of orientation-preserving homeomorphisms of $\wt{f}_r:\partial\D\to\partial\D$ with rotation number $\theta$, $\wt{f}$ has degree one and rotation number $\theta$.
\end{proof}

{\bf End of the Proof.} Let us assume by contradiction that $\theta$ belongs to $\cal H$. Then $\wt{f}$ is analytically linearizable on $\partial \D$. The linearizing map has an injective holomorphic extension to a neighborhood of $\partial \D$, and by analytic continuation this extension still conjugates $\wt{f}$ to a rotation in a neighborhood of $\partial\D$. Since $\wt{f}(z)=f(z)$ for $|z|>1$, $f$ has a stable domain containing an annulus of the form $1<|z|<1+\epsilon'$. Thus $P^p$ has a stable domain containing a neighborhood of $\realm{\Delta}$, hence a neighborhood of $\partial \Delta$. This contradicts the fact that $\partial{\Delta}$ is contained in the Julia set.

\section{Proof of Lemma~\ref{lem:oops}}
%

%

In Lemma~\ref{lem:oops}, we assume that only one critical value $v$ of $P$ belongs to $\bigcup_{k\geq 0} \realm{\Delta}^k$. Suppose there are critical points $c \in \apo{\Delta}^k$, $c' \in \apo{\Delta}^l$, with $l$ different or equal to $k$, and let us prove that $c=c'$. 
Note that $P(c)\in\realm{\Delta}^{k+1}$ and $P(c')\in\realm{\Delta}^{l+1}$ so necessarily $P(c')=P(c)=v$.
The sets $\realm{\Delta}^{k+1}$ and $\realm{\Delta}^{l+1}$ have a common point $v$, and since $\realm{\Delta}^{k} \supset P^{p-1}(\realm{\Delta}^{k+1})$ and $\realm{\Delta}^{l} \supset P^{p-1}(\realm{\Delta}^{l+1})$, the sets $\realm{\Delta}^{k}$ and $\realm{\Delta}^{l}$ also have a common point.
Consider the compact connected set
\[ K=\realm{\Delta}^{k+1}\cup\realm{\Delta}^{l+1}
\]
 and let $L$ be the component of $P^{-1}(K)$ that contains $\realm{\Delta}^{k}$ and $L'$ be the one that contains $\realm{\Delta}^{l}$. Then $L=L'$ because they have a point in common.

Let us prove that the set $K$ is full.
If $k=l$ this is immediate since $K=\realm{\Delta}^{k+1}$.
Otherwise by\footnote{More generally, if $K$ and $L$ are two full compact subsets of $\C$ and $K\cap L$ is full then $K\cup L$ is full. This can be proved for instance by studying the end of the Mayer--Vietoris exact homology sequence applied to the complement of $K$ and of $L$ in the Riemann sphere.} Corollary~\ref{cor:sep:2}, $\realm{\Delta}^{k+1}\cap \realm{\Delta}^{l+1}$ is a single point $z$ and there are two curves $\gamma$ and $\gamma'$ from $\infty$ to $z$ such that the curve $\gamma\cup \gamma'$ separates $\realm{\Delta}^{k+1}\setminus\{z\}$ from $\realm{\Delta}^{l+1}\setminus\{z\}$. In fact, this curve cuts the plane into two domains $A$, $B$ with $\realm{\Delta}^{k+1}\setminus\{z\} \subset A$ and $\realm{\Delta}^{l+1}\setminus\{z\}\subset B$. The curve minus the point $z$ is contained in the unbounded component of $\C\setminus K$. In particular, if $C$ is a bounded connected component of $\C\setminus K$, then $C$ is disjoint from $\gamma\cup\gamma'$ thus contained either in $A$ or in $B$. Hence $\partial C$ is contained either in $\realm{\Delta}^{k+1}$ or $\realm{\Delta}^{l+1}$. But then $C$ is a bounded component of $\C\setminus\realm{\Delta}^{k+1}$ or $\C\setminus\realm{\Delta}^{l+1}$, which is impossible.
Hence $\realm{\Delta}^{k+1}\cup\realm{\Delta}^{l+1}$ is full.

Since there is only one critical value, we can apply case~\eqref{item:l:2} of Lemma~\ref{lem:l} to $K$ and $L$ and deduce $c=c'$. \qed

\section{Proof of Theorem~\ref{thm:preciii}}\label{s:mainproof}

Recall that $\realm{\Delta}^k$ is defined as $\fil{\ov{\Delta^k}}$, where $\Delta^k=P^k(\Delta)$ and that $\apo{\Delta}^k$ is the connected component of $P^{-1}(\realm{\Delta}^{k+1})$ that contains $\realm{\Delta}^{k}$.
The assumption $n_0=1$ of Theorem~\ref{thm:preciii} means is that there is only one critical point $c'$ in $\bigcup_{k\ge 0}\apo{\Delta}^k$.
Without loss of generality, we may assume that $c'\in\apo{\Delta}^0$. 
In Section~\ref{subsec:aprop}, Definition~\ref{def:uv} we introduced simply connected open sets $U_0$ and $V_1$ containing $\realm{\Delta}^0$ and $\realm{\Delta}^1$ respectively, such that $P:U_0 \to V_{1}$ is a ramified covering with all critical values in $\realm{\Delta}^{1}$. 
We distinguish two cases:
\begin{itemize}
\item
First case: $P(c')\in\partial\Delta^1$. 
The critical value $P(c')\in V_1$ has only one preimage in $U_0$ by Lemma~\ref{lem:l} case~\eqref{item:l:2}. Recall that $\partial \Delta$ surjects to $\partial \Delta^1$ under $P$ (see the beginning of Section~\ref{subsec:dyn}).
Hence there must be at least one preimage of $P(c')$ on $\partial \Delta$. This implies that $c'\in\partial\Delta$.
\item 
Second case: $P(c')\notin\partial\Delta^1$. We will show this cannot occur according to the following plan:
 \begin{enumerate}
 \item We will first prove that $\realm{\Delta}$ is locally totally invariant under $P^p$ (Sections~\ref{subsec:tli}, \ref{subsec:sym}, \ref{sub:mutual} and~\ref{sub:mutual2} or its variant~\ref{sub:alt}). 
 \item We will then consider the external map associated to $(P^p,\Delta)$, i.e.\ we conjugate $P^p$ by the conformal map $\C\setminus\realm{\Delta} \to \C\setminus\ov{\D}$. We prove in Section~\ref{subsec:conj} that the external map is a degree $d$ covering of the circle for some $g\geq 2$. (See Section~\ref{subsec:extmap} for a general discussion of external maps. Note that the proof of Theorem~\ref{thm:preci} also makes use of the external map associated to $(P^p,\Delta)$ but there, we get a degree $1$ cover instead.)
 \item The local total invariance implies that we get a map defined in an annulus $1<|z|<1+\epsilon$, but this time the induced circle map has degree $>1$. We will then use a theorem of Mañé ensuring hyperbolicity of this circle map, provided there are no non-repelling cycles on the circle (Section~\ref{subsec:Mane})
 \item To check the absence of such cycles (Section~\ref{subsubsec:nrc}), we will use
the second hypothesis of Theorem~\ref{thm:preciii}, i.e.\ that $P$ has only $2$ critical values. 
\item Once we have this hyperbolicity, it follows that the map has a polynomial-like restriction to a neighborhood of $\realm{\Delta}$, with only one critical point (Lemma~\ref{lem:hyp-ana-top}). 
\item We can then apply Herman's theorem (or Lemma~\ref{lem:precii}) to reach a contradiction.
\end{enumerate}
\end{itemize}
The rest of Section~\ref{s:mainproof} deals with the second case: we have assumed $P(c') \notin \partial \Delta^1$ to arrive to a contradiction.
In particular, not all that we prove in the rest of this section needs to occur in reality. 

\subsection{Towards local invariance}\label{subsec:tli}
As $P^p(\realm{\Delta})\subset\realm{\Delta}$, we have $\realm{\Delta}\subset P^{-p}(\realm{\Delta})$.
 By Corollary~\ref{lem:locbwinv} and the discussion that follows it, to prove that \emph{$\realm{\Delta}$ is locally totally invariant by $P^p$} it is sufficient to prove that $\apo{\Delta}^k=\realm{\Delta}^k$ for every $k$.

\begin{lemma} $P(c')$ belongs to a component shielded by $\partial\Delta^1$ which is not $\Delta^1$.
\end{lemma}\begin{proof} The assumption that the critical value $P(c')$ does not belong to $\partial\Delta^1$ implies that it belongs to a component shielded by $\partial\Delta^1$. The restriction $P:\Delta\to\Delta^1$ is a holomorphic bijection, in particular it is surjective and has no critical point. Thus every point of $\Delta^1$ has at least one preimage in $U_0$ that is not a critical point. Now $P(c')$ has a unique preimage in $U_0$ (namely $c'$) and it is a critical point. Hence $P(c')\notin \Delta^1$.
Thus $P(c')$ belongs to component shielded by $\partial \Delta^1$ but different from $\Delta^1$.\end{proof}

\begin{lemma}\label{lem:tec}
If $k\neq 0\pmod p$ then $\apo{\Delta}^k = \realm{\Delta}^k.$
\end{lemma}
\begin{proof} We saw that the critical value $P(c')$ belongs to a component shielded by $\partial\Delta^1$.
Since on one hand $P(c')$ belongs to the Fatou set and on the other hand $\realm{\Delta}^i \cap \realm{\Delta}^j$ is either empty or a single point in the Julia set (Corollary~\ref{cor:sep}), it follows that $P(c')\notin\realm{\Delta}^k$ when $k\neq1\pmod p$, hence $\apo{\Delta}^{k-1}$ contains no critical point (recall the assumption that $c'$ is the only critical point in $\displaystyle \bigcup_{k\ge 0}\apo{\Delta}^k$). Thus by Corollary~\ref{cor:zero}:
$\apo{\Delta}^k = \realm{\Delta}^k \text{ if }k\neq 0\pmod p.$
\end{proof}
The next few sections aim to prove that this also holds when $k=0\pmod p$, i.e.\ that
\begin{equation}\label{eq:inv}
\apo{\Delta}=\realm{\Delta}.
\end{equation}

\begin{lemma}\label{lem:onepreim}There exists $m\geq 1$ such that 
$P^{mp}(c')\in P^{-1}(\Delta^1)\setminus \Delta$.
 \end{lemma}
 \begin{proof} The point $P(c')$ belongs to $\realm\Delta^1$. Moreover, since it is in a shielded component it is eventually mapped under iterations of $P$ to $\bigcup_k{\Delta}^k$ (Lemma~\ref{lem:rog}). 
Neither $c'$ nor $P(c')$ belongs to $\bigcup_k\Delta^k$. By Lemma~\ref{lem:tec} and Corollary~\ref{cor:zero}, $P$ is a bijection from $\realm{\Delta}^k$ to $\realm{\Delta}^{k+1}$ when $k\neq 0\pmod p$ and $P$ is also a bijection from $\Delta^k$ to $\Delta^{k+1}$. It follows that a point in $\realm{\Delta}^k$ that maps to $\Delta^{k+1}$ necessarily belongs to $\Delta^k$ if $k\neq 0\pmod p$.
Hence the first iterate of $c'$ that hits $\bigcup_k\Delta^k$ is a point $P^{k_0}(c')\in \Delta^1$ with $k_0=1+mp$ and $m\geq 1$. Thus $P^{mp}(c')\in \apo\Delta$ belongs to a preimage of $\Delta^1$ which is not $\Delta. $
\end{proof}

\subsection{Symmetries}\label{subsec:sym} To prove~(\ref{eq:inv}), we will use in an essential way the fact that there is some symmetry on the preimages of $\Delta^1$ in $\apo\Delta$ and in the global picture of $U_0$. 

The map $P:U_0\to V_1$ is a ramified covering with a single critical value $P(c')$ and $V_1$ is simply connected so it is equivalent to the map $D:\D\to\D$, $z\mapsto z^d$ where $d$ is the local degree of $P$ at $c'$.
The map $D:z\mapsto z^d$ is a normal covering\footnote{Recall that the \emph{group of automorphisms} or \emph{deck transformation group} of a covering $f:X\to Y$ is the set of homeomorphisms $\phi:X\to X$ so that $f\circ \phi=\phi$, and that the covering is termed \emph{regular}, \emph{normal} or \emph{Galois} whenever its action is transitive on fibers of $f$.} and has a finite cyclic group of automorphisms, namely the group of rotations generated by $z\mapsto e^{2\pi i/d}z$. Hence the ramified covering $P:U_0\to V_1$ is normal too and its group of autormophisms $G$ is isomorphic to $(\Z/d\Z,+)$. 

Let $\rho:U_0\to U_0$ be a generator of $G$. 
Since $P(c')$ does not belong to $\Delta^1$, this implies that $\Delta^1$ has exactly $d$ preimage components in $\apo{\Delta}$, the sets $\rho^n(\Delta)$, $n\in\Z$. Let us call them \emph{partners} of $\Delta$ and denote their union by $\cal O$: $$\cal O = \bigcup_{g\in G} g(\Delta).
$$
\begin{lemma}\begin{equation}\label{eq:O}
 \apo{\Delta} = \realm{\cal O}.
\end{equation}

\end{lemma}
\begin{proof}
Note that \[ \cal O = \apo{\Delta} \cap P^{-1}(\Delta^1) = U_0 \cap P^{-1}(\Delta^1) = \bigcup_{n\in\Z} \rho^n(\Delta) = \bigcup_{g\in G} g(\Delta).
\]
By Corollary~\ref{cor:preim}, $P^{-1}(\realm{\Delta^1})=\realm{P^{-1}(\Delta^1)}$.
The set $\apo{\Delta}$ is thus a connected component of $\realm{P^{-1}(\Delta^1)}$, so by Proposition~\ref{prop:chapo},
$\apo{\Delta} = \fil{\ov{P^{-1}(\Delta^1)} \cap \apo{\Delta}}$. Now $P^{-1}(\Delta^1)$ consists of finitely many components: the ones in $\apo{\Delta}$, whose union compose $\cal O$ and the rest, which are disjoint from $U_0$, so their closure is disjoint from $\apo{\Delta}$. This shows $\ov{P^{-1}(\Delta^1)} \cap \apo{\Delta} = \ov{P^{-1}(\Delta^1) \cap \apo{\Delta}}$, so $$\fil{\ov{P^{-1}(\Delta^1)} \cap \apo{\Delta}} = \fil{\ov{P^{-1}(\Delta^1) \cap \apo{\Delta}}} = \fil{\ov{\cal O}} = \realm{\cal O}.$$ Hence $
 \apo{\Delta} = \realm{\cal O}.$
\end{proof}


\subsection{Some topological considerations }\label{sub:mutual}

In this section, we reduce the proof of local total invariance to a general topological proposition. In the course of this argument, it will be shown that all partners of $\Delta$ share the same boundary. Together with the unbounded complementary component of $\partial \Delta$, they thus form a system of lakes of Wada.\footnote{Recall that we are in the middle of a proof by contradiction, so our work does not have any bearing on the problem of existence of lakes of Wada. Yet, Rogers proved this existence under the sole assumption that the boundary of the polynomial Siegel disk $\Delta$ has more than two complementary components (see Section~\ref{sec:state}) a hypothesis that has not been proven possible or impossible.}

\begin{lemma}\label{lem:u}
 $\partial \Delta$ shields one of the partners of $\Delta$ other than $\Delta$. In other words,\[ u(\Delta)\subset\realm{\Delta} \text{ for some } u\in G\setminus\{\on{id}\}.
\]
\end{lemma}
\begin{proof} By Lemma~\ref{lem:onepreim}, $P^{mp}(c')$ belongs to a component of $P^{-1}(\Delta^1)$ other than $\Delta$. Hence, $P^{mp}(c')\in u(\Delta)$ for some $u\in G$ with $u\neq \on{id}$.
Recall that $P(c')$ is contained in a component shielded by $\partial\Delta^1$. Since $mp\geq 1$, it follows that $P^{mp}(c')$ belongs to a component shielded by $\partial\Delta$. This proves $u(\Delta)\subset\realm{\Delta}$.
\end{proof}

Our next goal is to show that $\partial\Delta$ shields \emph{all} partners of $\Delta$. This can be rephrased as $\cal O \subset \realm{\Delta}$ where $\cal O=\bigcup_{g\in G} g(\Delta)$ was introduced above.
Once this is established, it follows from Corollary~\ref{cor:rincl} that $\realm{\cal O} \subset \realm{\Delta}$.
Thus, by Equation~\eqref{eq:O} above, 
\[ \realm{\Delta} \subset \apo{\Delta} = \realm{\mathcal O} \subset \realm{\Delta}
,\]
which proves the desired local total invariance (cf.\ Section~\ref{subsec:tli}).

The fact that $\partial\Delta$ shields all partners of $\Delta$ follows immediately from the following more general\footnote{We have not tried to find minimal hypotheses in Proposition~\ref{prop:wada}.} topological proposition, by going to the disk model of $U_0$ in which elements of $G$ act as rotations:

\begin{proposition}\label{prop:wada}
Let $U_0$ be a connected and simply connected open subset of $\C$, $G$ be a non-trivial finite cyclic group of homeomorphisms of $U_0$ that is conjugate to a group of rotations on the unit disk. Let $c'$ be the common fixed point
and let $\Delta$ be a non-empty connected and simply connected open subset of $U_0$ such that $g(\Delta) \cap \Delta = \emptyset$ for all $g\in G\setminus\{\on{id}\}$, such that $\partial \Delta = \partial \realm{\Delta}$ and such that $c'$ does not belong to $\partial\Delta$.
If there is a $u\in G\setminus\{\on{id}\}$ such that $u(\Delta)\subset\realm{\Delta}$ then for every $g\in G$, $g(\realm{\Delta})=\realm{\Delta}$.
\end{proposition}

For our case where $\Delta$ is the Siegel disk, the hypothesis $\partial \Delta = \partial \realm{\Delta}$ is satisfied according to Lemma~\ref{lem:oubli}. Moreover, By Lemma~\ref{lem:u} we get the assumption that there is a $ u\in G\setminus\{\on{id}\}$ such that $u(\Delta)\subset\realm{\Delta}$.
Note that the conclusion $g(\realm{\Delta})=\realm{\Delta}$ implies that $g({\Delta})\subset\realm{\Delta}$ for every $g \in G$ and therefore $\mathcal O\subset \realm{\Delta}$.

Section~\ref{sub:mutual2} is devoted to a self-contained proof of Proposition~\ref{prop:wada}. Section~\ref{sub:alt} gives a simpler proof of ${\mathcal O} \subset \widehat{\Delta}$ suggested by the referee, based on a theorem of Rogers.

\subsection{Proof of Proposition~\ref{prop:wada}}\label{sub:mutual2}
 
As this proof is a bit long, let us describe the plan: 
The proof goes by contradiction. Assume that there is some $g\in G$ such that $g(\realm\Delta)\neq \realm \Delta$. Then 
the stabilizer $H$ of $\realm{\Delta}$ in $G$ is a proper subgroup of $G$. We first prove that $u\in H$ so that $H$ is not trivial. 
Let $\rho$ be a generator of $G$, $\rho\notin H$. Then $\partial \Delta\neq \rho(\partial \Delta)$ but we prove that they both separate $c'$ from $\infty$. Intuitively, taking the images by iterates of $\rho$ should lead to a contradiction, but we have not been able to push this argument through using $\partial \Delta$ itself. Instead, we use two curves $\gamma$ and $\gamma'$ that are disjoint and such that each separate one another from $\infty$, which is an impossible configuration. We now begin the proof.
 
Recall that $G$ acts on $U_0$. The stabilizer (in $G$) of a set $A \subset U_0$ is defined by
\[\on{Stab} A =\setof{g\in G}{g(A)=A}
.\]
Let
\[ H= \on{Stab} \realm{\Delta}
.\]
Let $G\{\realm{\Delta}\}$ denote the orbit of $\realm{\Delta}$ under $G$:
\[G\{\realm{\Delta}\} = \{\realm{\Delta}, \rho(\realm{\Delta}), \rho^2(\realm{\Delta}), \ldots\}.\]
Do not confuse this with the the following notation that we also use below: for $X\subset U_0$, let
\[H\cdot X = \bigcup_{h\in H} h(X).\]

\begin{lemma}\label{lem:A}
 For every $M,N\in G\{\realm{\Delta}\}$, the inclusion $N \subset M$ implies $N=M$.
\end{lemma}
\begin{proof}
There is $g\in G$ such that $N=g(M)$, so $g(M)\subset M$. Let $d$ be the order of $G$.
Then $M = g^{d}(M) \subset g^{d-1}(M) \subset \cdots \subset g^{2}(M) \subset g(M) \subset M$. Thus $g(M)=M$.
\end{proof}

\begin{lemma}\label{lem:rgxgrx}
If $X$ is a compact subset of $U_0$ then every $g\in G$ is defined on $\realm{X}$ and $\realm{g(X)} = g(\realm{X})$.
\end{lemma}
\begin{proof}
 Recall that $\realm X$ is the union of $X$ and all bounded connected components of $\C\setminus X$.
 Also, $U_0$ is simply connected, so its complement in the Riemann sphere is connected. As a consequence, for any compact subset $X$ of $U_0$, the unbounded component of $\C\setminus X$ must contain $\C\setminus U_0$, hence $\realm{X}$ is contained in $U_0$. The latter is the domain of $g$, hence $g$ is defined on $\realm{X}$.
Consider any bounded component $C$ of $\C\setminus X$. Then $\partial C\subset X$, so $\partial g(C) = g(\partial C) \subset g(X)$. Moreover, since $g(C)$ is open, it must be a bounded component of the complement of $g(X)$, so $g(C)\subset \realm {g(X)}$.
Hence $g(\realm{X})\subset\realm {g(X)}$.
Replacing $g$ by $g^{-1}$ and $X$ by $g(X)$ in the above argument leads to the opposite inclusion.
\end{proof}

\begin{lemma}The stabilizer $H$ is non-trivial.
\end{lemma}\begin{proof}
Recall that $u(\Delta)\subset \realm{\Delta}$ for some $u\in G\setminus\{\on{id}\}$ by the assumption of in Proposition~\ref{prop:wada}. It follows from Corollary~\ref{cor:rincl} that $\realm{u(\Delta)}\subset\realm{\Delta}$.
By Lemma~\ref{lem:rgxgrx}, $u(\realm{\Delta})\subset \realm{\Delta}.$
Applying Lemma~\ref{lem:A}, we obtain
$u(\realm{\Delta})=\realm{\Delta}.$\end{proof}

 \begin{lemma}\label{lem:AH}
$H = \on{Stab} \realm{\Delta} = \on{Stab} \partial\Delta = \on{Stab} g(\realm{\Delta}) = \on{Stab} g(\partial \Delta)$ for every $g\in G$.
\end{lemma}
\begin{proof}
By the assumption $\partial \Delta = \partial \realm{\Delta}$. Thus, if $h \in H$, then $h(\partial \Delta)=h(\partial \realm{\Delta})=\partial h(\realm{\Delta})=\partial \realm{\Delta} = \partial \Delta$, which proves $h \in \on{Stab} \partial \Delta$. Conversely, if $h \in \on{Stab} \partial \Delta$, then $\partial h(\realm{\Delta})=h(\partial \realm{\Delta})=h(\partial \Delta)=\partial \Delta=\partial \realm{\Delta}$. By Lemma~\ref{lem:rgxgrx},
$h(\realm{\Delta})=\realm{h(\Delta)}$.
Thus $\realm{\Delta}$ and $\realm{h(\Delta)}$ have the same boundary. Since they are full, they must be equal. Hence $h(\realm{\Delta})=\realm{h(\Delta)}=\realm{\Delta}$ which proves $h \in H$. 

To finish the proof, note that for every $g \in G$, both stabilizers of $g(\realm{\Delta})$ and $g(\partial \realm{\Delta})$ coincide with the conjugate group $gHg^{-1}$, which is just $H$ since $G$ is commutative.\end{proof}

\begin{lemma}\label{lem:sepa2}
For every $g\in G$, $\partial g(\Delta)$ separates $c'$ from $\infty$.
\end{lemma}
\begin{proof}
 It is enough to prove that $\partial \Delta$ separates $c'$ from $\partial U_0$. Assume by way of contradiction that it does not. Then there is a path $\delta$ in $U_0$ from $c'$ to $\partial U_0$ and avoiding $\partial \Delta$.
  Let $V$ be the connected component of $U_0\setminus (H\cdot \delta)$  containing $\partial \Delta$. For every $h\in H$, since $h(\partial\Delta)=\partial\Delta$, the sets $V$ and $h(V)$ are not disjoint (thus are equal).
 We now work in coordinates in which the group $G$ is a group of rotations on a disk centered on $0$, where for simplicity we will use the same name for all the objects.
  Choose $h\in H\setminus\{\on{id}\}$ and $z\in V\cap h(V)$: since both $z$ and $h^{-1}(z)$ belong to $V$ and since $V$ is open and  connected,  there exists a path $\gamma$ in $V$ from $h^{-1}(z)$ to $z$. The winding number\footnote{By the winding number of an open path here we mean the difference of the initial and final values of a lift of the argument along the path, divided by $2\pi$.} around $0$ of this path is of the form $\alpha=i/d$ for some $i\in \Z$ and $i\notin d\Z$ because $h\neq \on{id}$ hence $\alpha\neq 0$. Let $k$ be the order of $h$. The concatenation of $\gamma$, $h(\gamma)$, $h^2(\gamma)$, \ldots, $h^{k-1}(\gamma)$ has winding number $k\alpha$ and is closed so $k\alpha\in\Z\setminus\{0\}$. Hence $\gamma'$ separates $0$ from $\infty$. Note that $H\cdot \delta$ is connected and disjoint from $\gamma'$. Hence $\gamma'$ separates $H\cdot\delta$ from the boundary of the round disk, which contradicts the fact that $\delta$ starts from $0$ and tends to this boundary.
 
\end{proof}


The generator $\rho$ of $G$ is not in $ H$, i.e.\ $\realm{\Delta} \neq \rho(\realm{\Delta})$, so by Lemma~\ref{lem:A}
\begin{equation}\label{eq:cen}
 \realm{\Delta}\not\subset \rho(\realm{\Delta})\text{ and }\rho(\realm{\Delta})\not\subset \realm{\Delta}.
\end{equation}
The first non-inclusion implies $\partial\Delta\not\subset\rho(\realm{\Delta})$, 
otherwise $\realm{\Delta} =\realm{\partial\Delta} \subset \realm{\rho(\realm \Delta)} = \rho\Big(\realm{\realm{\Delta}}\Big) = \rho(\realm{\Delta})$.
The second non-inclusion gives a similar conclusion, hence
\begin{equation}\label{eq:cen2}
 \partial\Delta\not\subset \rho(\realm{\Delta})\text{ and }\rho(\partial\Delta)\not\subset \realm{\Delta}.
\end{equation}

\begin{lemma}
\begin{equation}\label{eq:cen3}
\Delta\cap \realm{\rho(\Delta)} = \emptyset\text{ and }\rho(\Delta)\cap\realm{\Delta} = \emptyset.
\end{equation}
\end{lemma}
\begin{proof}
Since $\Delta$ and $\rho(\Delta)$ are disjoint open subsets of $U_0$, $\Delta$ does not intersect $\partial\rho(\Delta)$. If the first intersection were not empty, since $\Delta$ is connected, we would get $\Delta\subset\realm{\rho(\Delta)}$ and thus $\realm{\Delta} \subset \realm{\rho(\Delta)}$ by Corollary~\ref{cor:rincl}, leading to a contradiction. The second intersection is empty by a similar argument.
\end{proof}
 
 \smallskip

\noindent\textit{Construction of a pair of disjoint curves $\gamma$ and $\gamma'$.}

\smallskip

\begin{figure}
\begin{tikzpicture}
\node at (0,0) {\includegraphics[width=10cm]{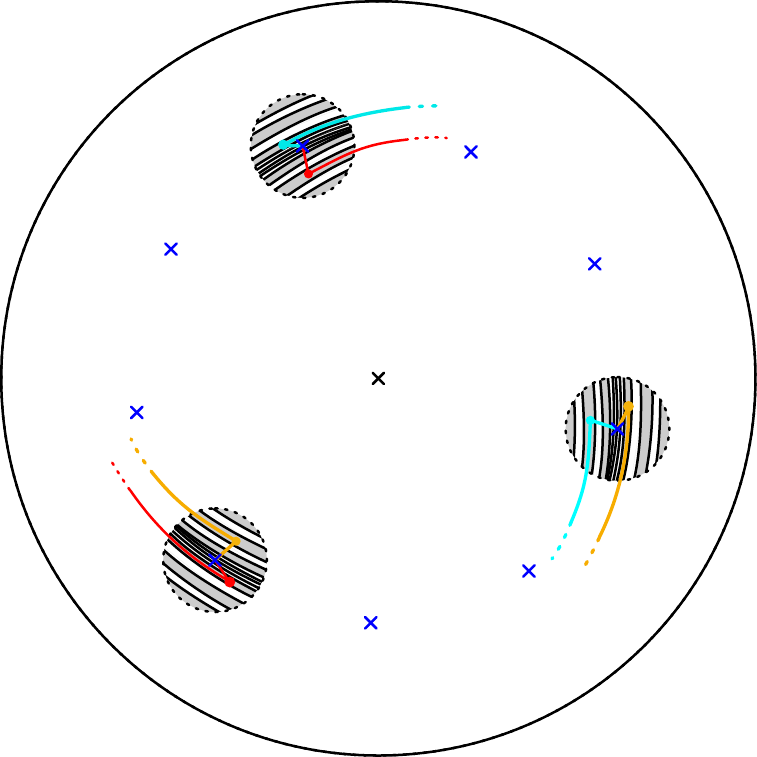}};
\node at (0.32,0.22) {$c'$};
\node at (4,4) {$U_0$};
\node at (-2.1,1.9) {$\rho(z_0)$};
\node at (-2.05,3.1) {$D$};
\node at (-2.2,-3.5) {$b(D)$};
\node at (3.2,0.45) {$b^2(D)$};
\draw[-] (-1,4) node[anchor=south] {$z_0$} -- (-1,3);
\draw[-] (-0.9,2) node[anchor=north] {$z_1$} -- (-0.9,2.7);
\draw[-] (-1.2,-2.7) node[anchor=west] {$z_2$} -- (-1.95,-2.7);
\node at (-3.7,-1.5) {$\gamma_0$};
\node at (-2.5,-1) {$b(\gamma_0)$};
\node at (0.5,2.8) {$\gamma_0$};
\node at (0.5,4) {$b^2(\gamma_0)$};
\node at (3.5,-2) {$b(\gamma_0)$};
\node at (2,-1.6) {$b^2(\gamma_0)$};
\end{tikzpicture}
\caption{Schematic illustration of the construction of the curve $\gamma$ in the case $|G|=9$ and $|H|=3$, in the disk model. Here $\gamma=\gamma_0\cdot b(\gamma_0)\cdot b^2(\gamma_0)$. The part of $\Delta$ that sits in $D=B(z_0,\epsilon)$ and its images under $H$ are indicated in gray, its boundary in black. The curve $\gamma_0$ is indicated in red, but only its beginning and its end are shown. Similarly $b(\gamma_0)$ is indicated in yellow and $b^2(\gamma_0)$ is indicated in cyan. The point $z_0$ and its orbit under $G$ are indicated by blue crosses.}
\label{fig:schema}
\end{figure}
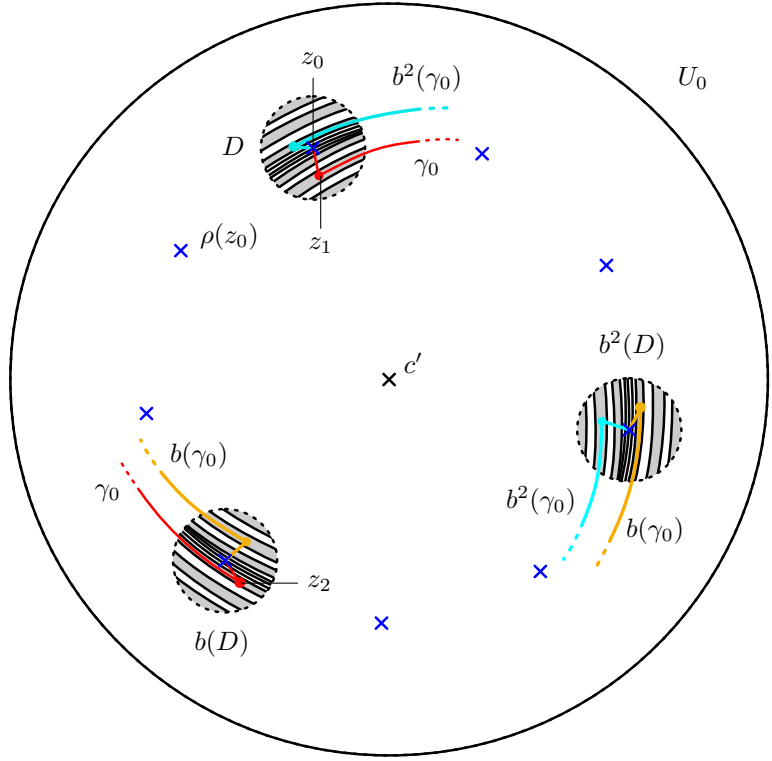

The first non-inclusion in \eqref{eq:cen2} means that there is a point $z_0\in\partial{\Delta}\setminus\rho(\realm{\Delta})$. Since $\rho(\realm{\Delta})$ is closed, there is a neighborhood of $z_0$ that does not intersect $\rho(\realm{\Delta})$. Let $D=B(z_0,\epsilon)$ be contained in this neighborhood and in $U_0$ and not containing $c'$ ($c'\neq z_0$ because $c'\notin\partial \Delta$ and $z_0\in \partial \Delta$):
\[D \subset U_0,\ D\cap \rho(\realm{\Delta})=\emptyset,\ c'\notin D.\]
 Since $z_0\in\partial\Delta$, $D$ intersects $\Delta$.
Denote by $b$ a generator of the group $H$. Since $b\in H$, $b(\partial \Delta)=\partial \Delta$ so $b(z_0)\in \partial \Delta$ thus $b(D)$ intersects $\Delta$ too.
Now choose any curve contained in $\Delta$ and going from a point $z_1$ in $\Delta\cap D$ to a point $z_2$ in $\Delta \cap b(D)$. Complete it by any curve contained in $D$ from $z_1$ to $z_0$ and by any curve contained in $b(D)$ from $z_2$ to $b(z_0)$: thus we get a curve $\gamma_0$ going from $z_0$ to $b(z_0)$.
Extend this curve to a closed\footnote{It is not hard to get a \emph{simple} closed curve, but we will not need to.} curve $\gamma = \gamma_0\cdot b(\gamma_0)\cdots b^k(\gamma_0)$ where $k\geq 2$ is the order of $H$. Notice that $\gamma$ does not contain the critical point $c'$.\footnote{One could naively think we could do without $D$ by choosing $z_0$ accessible from $\Delta$. But there is no guarantee that $b(z_0)$ is accessible from $\Delta$. The image under $b$ of the access to $z_0$ from $\Delta$ is an access to $b(z_0)$, but from $b(\Delta)$, not $\Delta$.}

To construct $\gamma$ we used the fact that $\realm{\Delta} \not\subset \rho(\realm\Delta)$. We similarly use $\rho(\realm{\Delta}) \not\subset \realm\Delta$ to construct a curve $\gamma'$. We can make the curves $\gamma$ and $\gamma'$ disjoint if we proceed as follows: We first select $z_0\in \partial \Delta\setminus \rho(\realm\Delta)$ and $z_0'\in\partial \rho(\Delta) \setminus \realm\Delta$. Note that $z_0$ and $z_0'$ cannot be in the same orbit under $H$ because $\partial \rho(\Delta)$ is invariant under $H$ and is contained in $\rho(\realm\Delta)$. We then choose $\epsilon$ and $\epsilon'$ small enough so that $D:=B(z_0,\epsilon)$ and $D':=B(z_0',\epsilon')$ satisfy
\bEA 
 &&D \subset U_0,\ D\cap \rho(\realm{\Delta})=\emptyset,\ c'\notin D,
 \\
 &&D' \subset U_0,\ D'\cap \realm{\Delta}=\emptyset,\ c'\notin D',
 \\
 &&D\text{ and }D'\text{ have disjoint orbits under }H.
\eEA
We then carry out the construction of $\gamma$ and $\gamma'$ as above. In particular:
\[z_0\in\gamma\cap\partial\Delta\text{ and }z'_0\in\gamma'\cap\partial\rho(\Delta).\]
Recall the following notation for $X\subset U_0$:
\[H\cdot X = \bigcup_{h\in H} h(X).\]
Recall that $H\cdot\Delta \subset \realm\Delta$ and so $H\cdot\rho(\Delta) \subset \rho(\realm\Delta)=\realm{\rho(\Delta)}$. In particular
 $H\cdot D$ is disjoint from $H\cdot\rho(\Delta)$. Now 
 $\gamma\subset (H\cdot\Delta) \cup (H\cdot D)$,
and
 $\gamma'\subset (\rho H\cdot\Delta) \cup (H\cdot D')$,
thus
\[ \gamma\cap\gamma'=\emptyset.
\]
Since on one hand $D$ is disjoint from $\realm{\rho(\Delta)} = H.\realm{\rho(\Delta)}$ and thus $H.D$ is disjoint from $\realm{\rho(\Delta)}$, and on the other hand $\Delta$ is disjoint from $\realm{\rho(\Delta)}$ by~\eqref{eq:cen3}, we also get $\gamma \cap \realm{\rho(\Delta)} = \emptyset$; similarly $\gamma' \cap \realm{\Delta} = \emptyset$:
\begin{equation}\label{eq:cen4}
\gamma \cap \realm{\rho(\Delta)} = \emptyset\text{ and }\gamma' \cap \realm{\Delta} = \emptyset.
\end{equation}

\begin{remark} It would be tempting to try to take $\gamma'=\rho(\gamma)$. But note that it is not even clear that one can take $z_0'=\rho(z_0)$: indeed, from $z_0\in\partial\Delta \setminus\rho(\realm\Delta)$ one deduces that $\rho(z_0) \in \partial \rho(\Delta)\setminus \rho^2(\realm\Delta)$, whereas we want $\rho(z_0) \in \partial \rho(\Delta)\setminus \realm\Delta$. Unless $\rho^2\in H$ (i.e.\ $H$ has index $2$ in $G$), there is no obvious reason for $\rho(z_0)$ to avoid $\realm{\Delta}$. 
\end{remark}

\begin{lemma}\label{lem:sepa}
Both $\gamma$ and $\gamma'$ separate $c'$ from $\infty$.
\end{lemma}
\begin{proof}
The proofs are similar so we only consider the case of $\gamma$. Let us go into the disk model of $U_0$: recall that by assumption there is a homeomorphism from $U_0$ to $\D$ that sends $c'$ to $0$ and conjugates the group $G$ to the group of rotations generated by $z\mapsto e^{2\pi i/d} z$.
The curve $\gamma_0$ joins $z_0$ to some $e^{2\pi i k/d} z_0 \neq z_0$. The lifted rotation angle around $0$ of $\gamma_0$ is therefore $\neq 0$. The lifted rotation angle of $\gamma$ is $d$ times the former, thus $\neq 0$. Therefore the winding number of $\gamma$ around $0$ is $\neq 0$.
So $\gamma$ separates $0$ from $\partial \D$ in the disk model. This proves the claim.
\end{proof}


\begin{lemma}\label{lem:noshare}
 Given two disjoint connected non-empty compact subsets $K$, $L$ of $\C$, one and only one of the following occurs:
 \begin{enumerate}[label=(\alph*)]
  \item\label{item:noshare:1} $\realm{K}\cap \realm{L}=\emptyset$,
  \item\label{item:noshare:2} $\realm{K}\subset\realm{L}\setminus L$,
  \item\label{item:noshare:3} $\realm{L}\subset\realm{K}\setminus K$.
 \end{enumerate}
\end{lemma}
\begin{proof}
The set $L$ is contained in a component $U$ of the complement of $K$,
the set $K$ is contained in a component $V$ of the complement of $L$.
Cases \ref{item:noshare:1}, \ref{item:noshare:2} and \ref{item:noshare:3} correspond respectively to the situations where $U$ and $V$ are unbounded, $U$ is unbounded but not $V$ and $V$ is unbounded but not $U$. The three cases are mutually exclusive.
There remains to rule out the last case where $U$ and $V$ are bounded. Consider the unbounded component $W$ of the complement of $K\cup L$. There is at least one point on $\partial W \subset K\cup L$ and if it belongs to $K$ then $V$ is unbounded, if it belongs to $L$ then $U$ is unbounded.
\end{proof}


We can apply Lemma~\ref{lem:noshare} to $K=\gamma$ and $L=\gamma'$.
Case~\ref{item:noshare:1} does not occur because of Lemma~\ref{lem:sepa}, thus either $\realm\gamma\subset\realm{\gamma'}\setminus\gamma'$ or $\realm{\gamma'}\subset\realm{\gamma}\setminus\gamma$. Let us assume that $\realm{\gamma'}\subset\realm{\gamma}\setminus\gamma$, the other case being treated similarly.
Let us apply Lemma~\ref{lem:noshare} to $K= \gamma'$ and $L=\partial \Delta$. By Lemmas~\ref{lem:sepa2} and~\ref{lem:sepa}, $\realm{\partial \Delta} \cap \realm{\gamma'}$ is non-empty since it contains $c'$.
By~\eqref{eq:cen4}, $\realm{\gamma'}$ cannot be contained in $\realm{\partial \Delta}$.
Thus the only remaining possibility is $\realm{\partial\Delta}\subset\realm{\gamma'}\setminus\gamma'$.
This, by the assumption above, leads to the inclusion
$\partial\Delta\subset\realm{\gamma}\setminus\gamma$. However, $\partial\Delta$ and $\gamma$ have a point in common (the point $z_0$ in the construction of $\gamma$). We reached a contradiction.

We have thus proved that $H\neq G$ leads to a contradiction. This ends the proof of Proposition~\ref{prop:wada}.


\subsection{Alternate proof}\label{sub:alt}

The referee suggested to us an alternate proof that all partners of $\Delta$ are shielded by $\partial\Delta$. In this proof, one does not need to use the topological Proposition~\ref{prop:wada} in full generality: by dealing with actual Siegel disks, the elaborate construction of $\gamma$ and $\gamma'$ is replaced by the use of the second theorem of Rogers mentioned in Section~\ref{sec:state}, more precisely the claim that $\partial\Delta$ is the boundary of each component of its complement, a fact that we did not use nor reprove here (see~\cite{Rogers}). This proof is much shorter, but comes at the cost of not being self-contained.

The proof begins as the proof of Proposition~\ref{prop:wada}: we introduce the stabilizer $H$ of $\partial \Delta$ and use the fact that $H$ is not trivial.
We also keep Lemma~\ref{lem:sepa2}: $\partial \Delta$ separates $c'$ from $\infty$. Then we note that the component $W$ of the complement of $\partial\Delta$ that contains $c'$ is $G$-invariant: indeed, $W$ is a Fatou component and for every $g\in G$, $g(W)$ is again a Fatou component ($g$ leaves invariant the intersection of the Julia and the Fatou set with $U_0$, since $P\circ g=P$) and contains $c'$ thus is equal to $W$.
By Rogers' theorem, $\partial W=\partial \Delta$, hence $\partial \Delta$ is $G$-invariant. The result follows.

\subsection{Analytic coverings of the circle with non vanishing derivative}\label{subsec:conj}

As explained in Section~\ref{subsec:extmap}, let $f$ be conjugate to $P^p$ by a conformal isomorphism $\phi : \C\setminus \realm{\Delta}\to \C\setminus\overline\D$ and $\wt f$ be its Schwarz reflection. 
Thanks to local total invariance, 
the domain of $\wt f$ contains some annulus of the form $1/(1+\epsilon)<|z|<1+\epsilon$. Recall that $d$ denotes the local degree of the critical point $c'\in U_0$ and that $\realm{\Delta}^k$ is defined as $\fil{\ov{\Delta^k}}$.

\begin{lemma}\label{lem:cover}
 The restriction of $\wt{f}$ to $\partial \D$ is a degree $d$ covering map of the circle with non-vanishing derivative.
\end{lemma}
\begin{proof}
 We saw in Section~\ref{subsec:extmap} that this restriction is an orientation-preserving covering of $\partial \D$ without critical points.
 Let $m$ be its degree.
 For every neighborhood $U$ of $\partial \D$ there exists a neighborhood $V$ of $\partial \D$ such that every point in $V\setminus\ov{\D}$ has exactly $m$ preimages under $f$ in $U\setminus \ov{\D}$.
The map $P$ is a degree $d$ ramified covering from $U_0$ to $V_1$ (whose critical value is in $\realm{\Delta}^1$) and an isomorphism from $U_k$ to $V_{k+1}$ when $k\neq 0\pmod p$. Thus for every neighborhood $U$ of $\realm{\Delta}^k$, there exists a neighborhood $V$ of $\realm{\Delta}^{k+1}$ such that every point of $V\setminus\realm{\Delta}^{k+1}$ has exactly $d$ or $1$ preimages under $P$ in $U\setminus\realm{\Delta}^k$, according as $k = 0 \pmod p$ or $k \neq 0 \pmod p$. From this it follows that for every neighborhood $U$ of $\realm{\Delta}$ there exists a neighborhood $V$ of $\realm{\Delta}$ such that every point in $V\setminus\realm{\Delta}$ has exactly $d$ preimages under $P^p$ in $U\setminus\realm{\Delta}$. Thus $m=d$.
\end{proof}

\subsection{Absence of non-repelling cycles on the circle for the external map}\label{subsubsec:nrc}

Assume by way of contradiction that the (extended) external map has a non-repelling cycle on $\partial {\mathbb D}$. Notice that since the unit circle is invariant, this cycle can only be attracting or parabolic with multiplier $1$.
Choose a point $a$ in the cycle, let $m$ be its period and let $\mathcal{A}$ be its immediate basin in the complex plane. Let $\mathcal{B}=\phi^{-1}(\mathcal{A}) = \phi^{-1}(\cal A\cap(\C \setminus \ov\D))$. (Recall that the domain of $\phi^{-1}$ is $\C\setminus\ov\D$.) Then $\mathcal{B}$ is bounded in $\C$ and stable under $P^{mp}$. It is therefore contained in a periodic component $\mathcal{B}'$ of the Fatou set.\footnote{It is not hard to show that in fact ${\mathcal B}' = {\mathcal B}$, but we do not need this.
}
Since every point of $\mathcal{A}$ tends to $a$ under the iteration of $\wt{f}^{m}$, every point of $\mathcal{B}$ tends to $\partial\realm{\Delta}=\partial\Delta$ under the iteration of $P^{mp}$.
In particular $\mathcal{B}'$ is not part of an attracting basin or a cycle of Siegel disks, because every point in such a component has an orbit that stays bounded away from the Julia set. It follows that $\cal B'$ is a component of an immediate parabolic basin. In particular the cycle of components associated to $\mathcal{B}'$ must contain a critical point.

Let us now use the assumption that $P$ has only two critical values. One of them is $P(c')$ and eventually maps to $\Delta$ so its orbit does not accumulate on $\partial \Delta$. The other must therefore be contained in the cycle of $\mathcal{B}'$ so its orbit accumulates on the associated parabolic cycle, a finite set. This contradicts Fatou's theorem asserting that $\partial\Delta\subset \bigcup_c \omega(c)$, the union being over all critical points.

Therefore there is no non-repelling cycle for $\wt f$ on $\partial \D$.

\begin{remark} For polynomials $P$ with more than two critical values, we have not been able to get a contradiction. The reason is that a third critical point may accumulate on all $\partial \Delta$ as required by Fatou's theorem. All our argument proves in this case is that if there is a non-repelling cycle for $\wt f$ on $\partial \D$ then there must be a parabolic cycle on $\partial \Delta$. Unfortunately the separation theorem does not rule this out: the separating external rays may well land on this parabolic point and still separate $\mathcal{B}'$ from $\Delta$. However, the notion of parabolic-like maps may help us deal with this possibility, see~\cite{luna}.
\end{remark}

This is the only place in the proof of Theorem~\ref{thm:preciii} where we use the assumption that $P$ has only two critical values.

\subsection{Hyperbolicity}\label{subsec:Mane}

Let us cite Theorem~A in \cite{Mane2} (see also \cite{Mane2b}).
\begin{theorem*}[Mañé]
Let $N=S^1$ (the circle) or $N=[0,1]$. If $f$ is a $C^2$ map from $N$ to $N$ and $\Lambda\subset N$ is a compact invariant set that doesn't contain critical points, sinks or non-hyperbolic periodic points, then either $\Lambda=N=S^1$ and $f$ is topologically equivalent to an irrational rotation or $\Lambda$ is a hyperbolic set.
\end{theorem*}

The second case does not exclude the possibility that $\Lambda=N=S^1$, consider for instance the angle doubling map on $S^1$.

We can apply Mañé's theorem to the extended external map $\wt{f}$ on $\partial \D$ with $\Lambda = N=\partial \D$: Section~\ref{subsubsec:nrc} proved the absence of attracting or parabolic cycles, and Lemma~\ref{lem:cover} the absence of critical points.
It follows that $\wt{f}$ is hyperbolic on $\partial \D$, which means that there is a continuous function $\rho>0$ on the unit circle such that $\wt{f}$ is uniformly expanding with respect to the metric $\rho(z)|dz|$, i.e.\ $\rho(\wt{f}(z))|\wt f'(z)|>\kappa \rho(z)$ for some $\kappa>1$. Let $s:z\mapsto 1/\ov z$.
\begin{lemma}\label{lem:hyp-ana-top}
There exists a pair of bounded topological annuli $U,V\subset\C$ such that $\partial \D\subset U \Subset V$, $s(U)=U$, $s(V)=V$ and $\wt f:U\to V$ is a covering map.
\end{lemma}
\begin{proof}This is a classical construction, we recall it for completeness.
The following fact will be used twice: \\
(*) Two inverse branches of a holomorphic map that coincide at a given point $z$ must coincide in a neighborhood of $z$. It follows that they must coincide in the connected component containing $z$ of the intersection of their domains.

Let $E$ be the holomorphic bijection from the cylinder $\C/\Z$ to $\C^*$ defined by $E(\zeta)=e^{2\pi i\zeta}$.
It sends $\R/\Z$ to $\partial \D$.
Let us make the coordinate change $z=E(\zeta)$, i.e.\ let us define $F=E^{-1}\circ \wt f\circ E$. To simplify notations below, let also $\mu(x)=\rho(E(x))$ for $x\in\R/\Z$. Then, since $|E'|=2\pi$ on the circle, we have
\[\mu(F(x))|F'(x)|>\kappa \mu(x).\]
Given $\epsilon>0$, let
\[r_x = \frac{\epsilon}{ \mu(x)}\]
and consider the domain
\[ B=\bigcup_{x\in\R/\Z} B_{x}\text{ with } B_{x} = B(x, r_x)
.\]
For $\epsilon$ small enough, all balls $B_{x}$ have radius $\leq 1/2$ and $B$ is contained in the domain of definition of $F$.

We will make use of the following quantitative version of the implicit function theorem:
\begin{sublemma}
Let $r,a>0$ and $f:B(0,r)\to\C$ be holomorphic with $|f'-1|<a<1$. Then
$f$ is injective, its image contains $B(f(0),(1-a)r)$, its inverse is holomorphic and $|(f^{-1})'-1|<a/(1-a)$.
\end{sublemma}
\noindent Injectivity follows from the fact that $f-\on{id}$ is contracting.
The second statement follows for instance from the fixed point theorem applied to the map $z\mapsto w+z-f(z)$, $w\in B(f(0),(1-a)r)$.
Details of the proof are left to the reader. 

Choose $\kappa'$ such that $$1<\kappa'<\kappa$$ and let
\[ a=1-\frac{1}{\kappa'}. \]
By uniform continuity of $F'$, provided $\epsilon$ is small enough we have
\[ |F'(x+\zeta)-F'(x)|<a|F'(x)| \text{ if } x\in\R/\Z  \text{ and } |\zeta|<r_x
.\]
Let us apply the sublemma above to $f_x:B(0,r_x)\to\C$ defined by $f_x(\zeta)= F(x+\zeta)/F'(x)$. Then $f_x$ has an inverse $g_x$ 
whose domain contains $B(F(x)/F'(x),(1-a)r_x)$. Consider the inverse branch $G_x$ of $F$ defined by $G_x(z)=x+g_x(z/F'(x))$.
Then the domain of $G_x$ contains $B(F(x),(1-a)r_xF'(x)) \Supset B(F(x),r_{F(x)}) = B_{F(x)}$. Since the domain of $G_x$ is equal to $F(B_x)$ it follows that
\[G_x(B_{F(x)}) \Subset B_x.\]
Set  $C_x=G_x(B_{F(x)})$
%
and
\[A=\bigcup_{x\in\R/\Z}C_x
.\]
The sublemma also states that $|g_x'|\leq 1+|g_x'-1|<1+a/(1-a)=1/(1-a)=\kappa'$ and thus
\[|G_x'(z)| \leq \kappa'/|F'(x)|,\]
hence
\[|G_x'(z)|\cdot r_{F(x)} \leq \frac{\kappa'}{\kappa}\cdot r_x,\]
whence
\[C_x\subset B(x,\frac{\kappa'}{\kappa} r_x).\]
The radii of $B(x,\frac{\kappa'}{\kappa}r_x)$ and $B_x$ differ by at least $(1-\frac{\kappa'}{\kappa})\cdot r_x$, whose infimum is positive, so  $A \Subset B$.

Let us now prove that $F:A\to B$ is a covering. For this, it is enough to show
that for every $x\in \R/\Z$, $A\cap F^{-1}(B_{x})$ is the disjoint union of $C_{x_1}$, \ldots, $C_{x_d}$ where $x_1$, \ldots, $x_d$ are the $d$ real preimages of $x$.\footnote{In other words, we prove that $B_x$ is an \emph{evenly covered} neighborhood of $x$.} It is immediate that $A\cap F^{-1}(B_{x})$ contains $C_{x_1} \cup \cdots \cup C_{x_d}$.
The union is on pairwise disjoint sets, for if $C_{x_i}$ meets $C_{x_j}$ then  by (*) the two branches $G_{x_i}$ and $G_{x_j}$ are equal.
There remains to prove that $A\cap F^{-1}(B_{x})$ is contained in $C_{x_1} \cup \cdots \cup C_{x_d}$.
Consider $\zeta\in B_{x}$ and any $\zeta'\in A$ such that $F(\zeta')=\zeta$. By definition of $A$, there exists $x'\in\R/\Z$ such that $\zeta'\in C_{x'}$, so $\zeta\in B_{F(x')}$ and $G_{x'}(\zeta)=\zeta'$. Let $x''=\Re(\zeta)$. The vertical line segment $[\zeta,x'']$ is contained in $B_{x} \cap B_{F(x')}$. The point $x''$ has exactly $d$ preimages, so
there exists $i\in\{1,\ldots,d\}$ such that $G_{x_i}(x'')=G_{x'}(x'')$. Then by (*) the maps $G_{x'}$ and $G_{x_i}$ coincide in a neighborhood of the segment.
Hence $G_{x_i}(\zeta) = G_{x'}(\zeta)$ and $\zeta'\in C_{x_i}$.

The open set $B$ is real-symmetric and is a topological annulus: indeed its intersection with every vertical is a connected open interval containing a real point. Since $F:A\to B$ is a finite degree covering, $A$ is also a topological annulus.

The pair of sets $U=E(A)$ and $V=E(B)$ satisfies the conclusions of the lemma.
\end{proof}

Let $U'=\realm{\Delta}\cup \phi^{-1}(U)$ and $V'=\realm{\Delta}\cup \phi^{-1}(V)$.
Then $U'$ and $V'$ are open, connected, simply connected, $U'$ is a connected component of $P^{-p}(V')$ and $U'$ is compactly contained in $V'$. Therefore the restriction $P^p: U'\to V'$ is a polynomial-like map, which is also unicritical.

So we have proved that $P^p$ has a unicritical polynomial-like restriction whose domain contains $\Delta$. For this restriction, $\Delta$ is a Siegel disk of rotation number $\theta$ and period one. Now we can apply Herman's theorem (which works as well for unicritical polynomial-like maps) to conclude that the critical point $c'$ must belong to $\partial \Delta$. This contradicts our standing assumption that $P(c')\notin\partial\Delta^1$. Alternatively, we can finish the argument as in Lemma~\ref{lem:precii}: Fatou's theorem stating that $\partial \Delta$ is contained in the $\omega$-limit set of the critical points also works for polynomial-like maps, and since the only critical point $c'$ of our polynomial-like map is in a Fatou component, the whole boundary 
$\partial \Delta$ cannot be accumulated by its orbit, leading to a contradiction.

This finishes the proof of the second case of Section~\ref{s:mainproof}: the assumption $P(c')\notin\partial\Delta^1$ leads to a contradiction.

The proof of Theorem~\ref{thm:preciii} is thus complete.

\section*{Acknowledgements}
The authors would like to thank the referee for a very detailed review, leading to notable improvements of the text. We also thank Anna Miriam Benini and Núria Fagella for useful discussions that also improved our text.
Both authors were partially funded by the grant \emph{Lambda} of Agence Nationale de la Recherche: ANR-13-BS01-0002.

\bibliographystyle{alpha}
\bibliography{CritPtBdy}

\end{document}